\documentclass[12pt]{amsart}
\usepackage{geometry}   %????????
\usepackage[colorlinks,citecolor = red, linkcolor=blue,hyperindex]{hyperref}
\usepackage{euscript,eufrak,verbatim, mathrsfs}
\usepackage[psamsfonts]{amssymb}
\usepackage{bbm}
\usepackage{graphicx}
 \usepackage{float}
\usepackage{float, tikz}

 \usepackage[all, cmtip]{xy}

\usepackage{upref, xcolor, dsfont}
\usepackage{amsfonts,amsmath,amstext,amsbsy, amsopn,amsthm}
\usepackage{enumerate}

\usepackage{url}

\usepackage{mathtools}
\usepackage{bookmark}

 \usepackage{euscript}
\usepackage{helvet}         % selects\textbf{\textbf{•}} Helvetica as sans-serif font
\usepackage{courier}        % selects Courier as typewriter font
\usepackage{type1cm}        % activate if the above 3 fonts are
\usepackage{multicol}        % used for the two-column index
\usepackage[bottom]{footmisc}% places footnotes at page bottom

\newtheorem{theorem}{Theorem}[section]
\newtheorem*{theorem*}{Theorem B}
\newtheorem{lemma}[theorem]{Lemma}

\newtheorem{proposition}[theorem]{Proposition}
\newtheorem{corollary}[theorem]{Corollary}

\newtheorem*{definition*}{Definition}

\newtheorem*{observation*}{Observation}

\newtheorem*{assumption*}{Assumption}
\newtheorem*{question*}{Question}
\newtheorem*{problem*}{Problem}

\theoremstyle{remark}

\newtheorem*{remark*}{Remark}

\newcommand{\R}{\mathbb{R}}
\newcommand{\N}{\mathbb{N}}

\newcommand{\E}{\mathbb{E}}

\newcommand{\V}{\mathcal{V}}

\newcommand{\spann}{\mathrm{span}}

\newcommand{\supp}{\mathrm{supp}}
\newcommand{\sgn}{\mathrm{sgn}}

\newcommand{\an}{\text{\, and \,}}

\begin{document}

\title[Metric projections onto convex cones]{The metric projections onto closed convex cones in a Hilbert space}

\author%[authorlabel1]
{Yanqi Qiu}
\address%[authorlabel1]
{Yanqi QIU: Institute of Mathematics and Hua Loo-Keng Key Laboratory of Mathematics, AMSS, Chinese Academy of Sciences, Beijing 100190, China.}
\email{yanqi.qiu@amss.ac.cn, yanqi.qiu@hotmail.com}

\author{Zipeng Wang}
\address{Zipeng WANG: College of Mathematics and Statistics, Chongqing University, Chongqing,
401331, P.R.China}
\email{zipengwang2012@gmail.com, zipengwang@cqu.edu.cn}

\thanks{Y. Qiu is supported by grants NSFC Y7116335K1,  NSFC 11801547 and NSFC 11688101 of National Natural Science Foundation of China. Z. Wang is supported by NSFC 11601296 of National Natural Science Foundation of China.}

\begin{abstract}
We study the metric projection onto  the closed convex cone in a real Hilbert space $\mathscr{H}$ generated by a sequence $\mathcal{V} = \{v_n\}_{n=0}^\infty$. The first main result of this paper provides a sufficient condition under which we can identify the closed convex cone generated by  $\mathcal{V}$  with the following set: 
\[
\mathcal{C}[[\mathcal{V}]]:  = \bigg\{\sum_{n=0}^\infty a_n v_n\Big|a_n\geq 0,\text{ the series }\sum_{n=0}^\infty a_n v_n\text{ converges in $\mathscr{H}$}\bigg\}.
\]
Then, by adapting classical results on general convex cones,  we give a useful description of the metric projection of a vector onto $\mathcal{C}[[\V]]$. As applications, we obtain the best approximations of many concrete functions in $L^2([-1,1])$ by polynomials with non-negative coefficients.
\end{abstract}


\maketitle

\setcounter{equation}{0}

\section{Introduction}

\subsection{Closed convex cone generated by a sequence of vectors}

Recall that a non-empty subset of a vector space over the field $\R$ of real numbers is a {\it convex cone} if it is closed under addition and closed under multiplication by non-negative scalars.  It is a classical problem to find the best approximation of a given vector by elements in a given closed convex subset, cf. \cite{Deutsch-2001,Domokos-2017,Edwards-1965,Jurkat-Lorentz-1961,
Roger-Bertram-1998,Toland-1996,Zarantonello-1971}.  The best approximation of functions by polynomials with non-negative coefficients on certain interval is particularly interesting, for instance,  it plays a crucial role in the spectral analysis of self-adjoint operators on real Hilbert spaces \cite{Roger-Bertram-1998,Toland-1996}.

In this paper, we are interested in the convex cone generated by a given sequence in a real Hilbert space.  More precisely,  let $\mathscr{H}$ be a Hilbert space over $\R$ and let $\V = \{v_n\}_{n=0}^{\infty}$ be a sequence in $\mathscr{H}$. The convex cone $\mathcal{C}[\V]$ generated by the sequence $\V$ is the set of all non-negative linear combinations of the vectors in $\V$ (here and after, we denote by $\N$ the set of all non-negative integers: $\N= \{0, 1, 2, \cdots\}$):
\[
\mathcal{C}[\V]:=\bigg\{\sum_{n=0}^N a_n v_n\Big|N \in \N, a_n\geq 0 \bigg\}.
\]
Denote by $\overline{\mathcal{C}}[\V]$ the closure of $\mathcal{C}[\V]$ inside $\mathscr{H}$. Then  $\overline{\mathcal{C}}[\V]$ is again a convex cone and moreover is closed.  The convex cone $\mathcal{C}[\V]$ and its closure $\overline{\mathcal{C}}[\V]$ are useful objects in functional analysis, mathematical optimization and many other fields, cf. \cite{Asimow-1980,Beck-2020,Edwards-1965,Stoer-1970}. One can consult \cite{Davis-1954,McKinney-1962} for the basic algebraic theory of $\mathcal{C}[\V]$ when $\V$ is a finite set of a finite dimensional Euclidean space and of an infinite dimensional Hilbert space respectively.  Usually, it is more convenient to work with the closure  $\overline{\mathcal{C}}[\V]$. In fact,  the closedness of a convex subset is important in the best approximation theory \cite[Chapter 5]{Deutsch-2001}.

 For any element $w\in \mathscr{H}$, let $d(w, \overline{\mathcal{C}}[\V])$ denote the distance of $w$ to  $\overline{\mathcal{C}}[\V]$:
\[
d(w,\overline{\mathcal{C}}[\V]):= \inf\bigg\{\|w-u \|\Big| u \in \overline{\mathcal{C}}[\V] \bigg\} = \inf\bigg\{\Big\|w-\sum_{n=0}^N a_nv_n\Big\|\Big|N\in \N,  a_n\geq 0\bigg\}.
\]
For a non-zero vector $w\in \mathscr{H}$, we also introduce  the relative distance $\lambda(w, \overline{\mathcal{C}}[\V])$:
\[
\lambda(w, \overline{\mathcal{C}}[\V]): = \frac{d(w, \overline{\mathcal{C}}[\V])}{\|w\|} \in [0, 1].
\]
Clearly, $\lambda(w, \overline{\mathcal{C}}[\V])= 1$ if and only if $P_{\overline{\mathcal{C}}[\V]}(w)=0$, while $\lambda(w, \overline{\mathcal{C}}[\V])= 0$ if and only if $w \in \overline{\mathcal{C}}[\V]$.  Denote by $\angle(w, v)$ the angle between two non-zero vectors $w,v$. If $\lambda(w, \overline{\mathcal{C}}[\V])< 1$, then  it is easy to see that it satisfies the equality
\[
\lambda(w, \overline{\mathcal{C}}[\V]) = \sin\left(\inf \bigg\{ \angle (w, v): v\in \overline{\mathcal{C}}[\V] \setminus \{0\} \bigg\}\right).
\]
In other words, the quantity $\lambda(w, \overline{\mathcal{C}}[\V])$, when belongs to the open interval $(0,1)$, mesures how far the direction of the vector $w$ is away from the directions of the vectors in $\overline{\mathcal{C}}[\V]$.

By a classical result on closed convex subsets of Hilbert spaces (see e.g. \cite[p. 239]{Zarantonello-1971} and \cite[Theorem 3.5]{Deutsch-2001}), for any $w\in \mathscr{H}$, there exists a {\it unique} $w^*\in \overline{\mathcal{C}}[\V]$ closest to $w$:
\[
\| w - w^*\| =  d(w, \overline{\mathcal{C}}[\V]).
\]
 This unique element $w^*$ is   called the {\it metric projection} of $w$ onto the closed convex cone $\overline{\mathcal{C}}[\V]$ and will be denoted by $P_{\overline{\mathcal{C}}[\V]}(w)$. In most situations, the computation of $d(w, \overline{\mathcal{C}}[\V])$ or $\lambda (w, \overline{\mathcal{C}}[\V])$ is then reduced to the characterization of   $P_{\overline{\mathcal{C}}[\V]}(w)$.

For studying the best approximation of a given vector by elements in $\overline{\mathcal{C}}[\V]$, one may first try to  understand better the closed convex cone $\overline{\mathcal{C}}[\V]$. However, in general, $\overline{\mathcal{C}}[\V]$ can be quite complicated. Therefore, in some particular cases,  it seems to be of independent interests to find an explicit description of $\overline{\mathcal{C}}[\V]$.  The first aim of this paper is to present an explicit description of $\overline{\mathcal{C}}[\V]$ under some additional assumptions on  $\V$.

Let us introduce a subset $\mathcal{C}[[\V]]$ as follows:
\[
\mathcal{C}[[\V]]:=\bigg\{\sum_{n=0}^\infty a_n v_n\Big|\text{$a_n\geq 0$ and the series \,}\sum_{n=0}^\infty a_n v_n\text{ converges in $\mathscr{H}$}\bigg\}.
\]
 We shall also denote $\mathcal{C}[[\V]]$ by $\mathcal{C}_{\mathscr{H}}[[\V]]$ when it is necessary. Clearly, we have 
\begin{align}\label{p-s-c}
\mathcal{C}[\V] \subset  \mathcal{C}[[\V]] \subset\overline{\mathcal{C}}[\V].
\end{align}
\begin{remark*}
In general,  the definition of $\mathcal{C}[[\V]]$ depends on the order of the vectors in the sequence $\V$. Indeed, the requirement that  the series $\sum_{n=0}^\infty a_n v_n$ converges in $\mathscr{H}$ in general does not imply that it converges unconditionally and hence the limit may depend on the order.
\end{remark*}

The  subset $\mathcal{C}[[\V]]$ defined as above is again a convex cone and by \eqref{p-s-c}, it is closed if and only if $\mathcal{C}[[\V]] =  \overline{\mathcal{C}}[\V]$.
 Therefore, if $\mathcal{C}[[\V]]$ is closed, then we obtain a  relatively simple description of $\overline{\mathcal{C}}[\V]$: all elements of $\overline{\mathcal{C}}[\V]$ are given as a convergent series whose terms are multiples of elements in $\V$.

However, in general $\mathcal{C}[[\V]]$ is not closed. For instance,  consider the two-dimensional Euclidean space $\R^2$ and let $v_n = (1, n)\in \R^2$ for $n \in\N$.  Then $\overline{\mathcal{C}}[\V] = \{(x,y)| x, y \ge 0\}$. But $\mathcal{C}[[\V]] \neq \overline{\mathcal{C}}[\V]$ since $(0, 1) \notin \mathcal{C}[[\V]]$.

So we  are going to investigate the following
\begin{problem*}
When is $\mathcal{C}[[\V]]$ closed ?
\end{problem*}

It is easy to see that $\mathcal{C}[[\V]]$ is closed if the sequence $\V$ satisfies the  condition: there exist two constants $c, C>0$ and a sequence of positive numbers $\lambda_n> 0$ such that the inequalities
\[
c \Big(\sum_{n=0}^\infty   \lambda_n |c_n|^2\Big)^{1/2} \le   \Big\|  \sum_{n= 0}^\infty  c_n v_n \Big\| \le C \Big(\sum_{n=0}^\infty   \lambda_n |c_n|^2\Big)^{1/2}
\]  hold for all finitely supported sequences $\{c_n\}_{n=0}^\infty$ in $\R$. A less obvious sufficient condition for  $\mathcal{C}[[\V]]$ to be closed is given in the following Theorem  \ref{prop-closed}.

Given a positive Radon measure $\mu$ on $\R_{+} = [0, \infty)$, we denote by $\supp(\mu)$ the topological support of $\mu$ and by $s_\mu$  the supremum of $\supp(\mu)$:
\[
s_\mu := \sup\{ x| x\in \supp (\mu)\} \in [0, \infty].
\]
We say that $\mu$ has finite moments of all orders if $\int_{\R_{+}} t^n d\mu(t)<\infty$ for all $n \in \N$.

\begin{theorem}\label{prop-closed}
Let $\mu$  be a positive Radon measure on $\R_{+}$ having finite moments of all orders and satisfying the condition:
\[
\mu\big(\R_{+} \setminus [0, s_\mu)\big) = 0.
\]
 Let  $\mathcal{W}   = \{w_n\}_{n =0}^\infty$ be a sequence  in another real Hilbert space $\mathscr{K}$ and assume that there exists a constant $C>0$ such that  the inequality
\begin{align}\label{ab-cond}
\Big\| \sum_{n = 0}^\infty  a_n w_n \Big\| \le  C \Big\|\sum_{n = 0}^\infty a_n t^n\Big\|_{L^2(\mu)}
\end{align}
holds for all finite supported sequences $\{a_n\}_{n=0}^\infty$ of  non-negative real numbers.  Assume that there exists a sequence  $\{\lambda_n\}_{n=0}^\infty$ of positive numbers such that the sequence $\V = \{v_n\}_{n= 0}^\infty \subset \mathscr{H}$ satisfies
\begin{align}\label{v-n-inner}
\langle \lambda_n v_n, \lambda_m v_m\rangle_{\mathscr{H}} = \int_\R t^{m+n} d\mu(t) + \langle w_n, w_m\rangle_{\mathscr{K}}   \quad \text{for all $n, m \in \N$}.
\end{align}
Then we have $\overline{\mathcal{C}}[\V] =\mathcal{C}[[\V]]$ and thus $\mathcal{C}[[\V]]$ is a closed convex cone.
\end{theorem}

\begin{remark*}
The condition $\mu\big(\R_{+} \setminus [0, s_\mu)\big) = 0$ means that either $s_\mu = \infty$ or  $\mu(\{s_\mu\}) = 0$ if $s_\mu<\infty$. This condition in general can not be removed in Theorem \ref{prop-closed}. For instance, let $\nu$ be the Lebesgue measure on $[0,1]$ and let $\mu= \nu + \delta_1$, where $\delta_1$ is the Dirac mass at the point $1$, then  the set
\[
\bigg\{\sum_{n=0}^\infty a_nt^n\Big|a_n\geq 0,\text{ the series }\sum_{n=0}^\infty a_nt^n\text{ converges in }L^2(\mu)\bigg\}
\]
is not closed. Indeed, the sequence $\{t^n\}_{n=0}^\infty$ converges in $L^2(\mu)$ to the Dirac function $\delta_1 \in L^2(\mu)$. But clearly, this limit function $\delta_1$ is not of the form $\sum_{n=0}^\infty a_n t^n$.
\end{remark*}

\begin{remark*}
By modifying the proof of Theorem \ref{prop-closed}, we can replace the sequence of functions $\{t^n\}_{n =0}^\infty$ by any sequence $\{h_n(t)\}_{n=0}^\infty$ of continuous non-decreasing non-negative functions on $\R_{+}$ satisfying the property:
\[
\sum_{n= 0}^\infty \frac{h_n(t)}{h_n(s)}<\infty  \quad \text{for any pair $(s, t)$ with $0\le t < s$.}
\]
\end{remark*}

Theorem \ref{prop-closed} has the following useful corollary.  Before stating the corollary, let us note that if  there exists a constant $C>0$ such that
\[
\langle w_n, w_m\rangle\le C \int_{\R_{+}} t^{m+n}d\mu(t) \quad \text{for all $m, n \in \N$},
\]
then the condition \eqref{ab-cond} is satisfied with the constant given by $\sqrt{C}$.

\begin{corollary}\label{cor-real-line}
Let $\nu$  be a positive Radon measure on $\R$.  Assume that the restriction  $\mu = \nu|_{\R_{+}}$ of the measure $\nu$ on $\R_{+}$ has  finite moments of all orders and satisfies the condition $\mu\big(\R_{+} \setminus [0, s_\mu)\big) = 0$. Assume moreover that there exists a constant $C>0$ such that
\begin{align}\label{neg-less-pos}
\int_{\R_{-}} t^{2n} d\nu(t) \le C \int_{\R_{+}} t^{2n} d\nu(t) \quad \text{for all $n\in \N$},
\end{align}
where $\R_{-}= \R\setminus \R_{+}$.  Then
\[
\mathcal{C}_{L^2(\nu)}[[ \{t^n\}_{n=0}^\infty]] = \bigg\{\sum_{n=0}^\infty a_nt^n\Big|a_n\geq 0,\text{ the series }\sum_{n=0}^\infty a_nt^n\text{ converges in }L^2(\nu)\bigg\}
\]
is a closed convex cone in $L^2(\nu)$.
\end{corollary}

\begin{remark*}
In general, the condition  \eqref{neg-less-pos} can not be removed in Corollary \ref{cor-real-line}. For instance, consider the Lebesgue measure on the interval $[-1, 0]$ and the associated Hilbert space $L^2([-1,0])$. Then the set
\[
\mathcal{C}_{L^2([-1,0])}[[ \{t^n\}_{n=0}^\infty]] = \bigg\{\sum_{n=0}^\infty a_nt^n\Big|a_n\geq 0,\text{ the series }\sum_{n=0}^\infty a_nt^n\text{ converges in }L^2([-1,0])\bigg\}
\]
is not closed in $L^2([-1,0])$. See the Appendix of this paper for the details.
\end{remark*}

\subsection{Metric projection onto a closed convex cone} Assume that the convex cone $\mathcal{C}[[\V]]$ is closed, that is $\mathcal{C}[[\V]] = \overline{\mathcal{C}}[\V]$. Proposition \ref{prop-det-proj} below is an application to our situation of the classical results (cf. \cite[Lemma 1.1]{Zarantonello-1971}) on the metric projection onto a closed convex cone.  We shall see that Proposition \ref{prop-det-proj} can be useful in computing explicitly the metric projections of given vectors.
  
\begin{definition*}
We say that a sequence  $\V= \{v_n\}_{n=0}^\infty$ in a Hilbert space $\mathscr{H}$ has no  positive relations, if the coincidence of two convergent series
\[
\sum_{n= 0}^\infty a_n v_n   = \sum_{n=0}^\infty b_n v_n
\]
with all coefficients  non-negative implies $a_n = b_n$ for all $n \in \N$.
\end{definition*}
\begin{remark*}
Note that if  the sequence $\V= \{v_n\}_{n=0}^\infty$ has no positive relations, then the vectors $v_n$'s are linearly independent. 
\end{remark*}

By convention, in what follows, we set
\[
\sum_{n\in \emptyset} a_n v_n := 0.
\]
\begin{proposition}\label{prop-det-proj}
Let $\V = \{v_n\}_{n =0}^\infty \subset \mathscr{H}$ be a sequence without positive relations and assume that $\mathcal{C}[[\V]]$ is closed. Then for any $w\in \mathscr{H}$, there exists a unique subset $S\subset \N$ such that
\begin{align}\label{P-CV-S}
P_{\mathcal{C}[[\V]]} (w) = \sum_{n\in S} a_n v_n \quad \text{with $a_n > 0$ for all $n\in S$,}
\end{align}
where $(S, \{a_n\}_{n\in S})$ is uniquely determined by
\begin{align}\label{unique-S-a}
\left\{
  \begin{array}{l}
    \displaystyle \sum_{n\in S} a_n v_n \text{\,converges in $\mathscr{H}$ and $a_n >0$ for all $n\in S$};  \\
   \displaystyle  \sum_{k\in S} a_k\langle v_k, v_n\rangle \ge \langle w, v_n\rangle \text{\, for all $n\in \N$};  \\
     \displaystyle  \sum_{k\in S} a_k\langle v_k, v_n\rangle  =  \langle w, v_n\rangle \text{\, for all $n\in S$}.
  \end{array}
\right.
\end{align}
\end{proposition}
\begin{remark*}
Note that in Proposition \ref{prop-det-proj}, by saying that $\sum_{n\in S} a_n v_n$ converges in $\mathscr{H}$, we mean that the following limit exists in $\mathscr{H}$:
\[
\lim_{N\to\infty} \sum_{n\in S, \, n\le N} a_n v_n.
\]
\end{remark*}

For any subset $S\subset \N$, define a subset $\mathscr{H}(\V, S)\subset \mathscr{H}$ by
\begin{align}\label{def-H-S}
\mathscr{H}(\V, S): = \bigg\{ w \in \mathscr{H} \Big| P_{\mathcal{C}[[\V]]} (w) = \sum_{n\in S} a_n v_n \quad \text{with $a_n > 0$ for all $n\in S$} \bigg\}.
\end{align}
In particular, we have  $\mathscr{H}(\V, \emptyset)= (P_{\mathcal{C}[[\V]]})^{-1}(0)$.

By noting that the conditions in \eqref{unique-S-a} are stable under addition and multiplication by a positive constant, we obtain the following corollary of Proposition~\ref{prop-det-proj}.

\begin{corollary}
Let $\V = \{v_n\}_{n =0}^\infty \subset \mathscr{H}$ be a sequence without positive relations and assume that $\mathcal{C}[[\V]]$ is closed. Then we have a partition of the whole Hilbert space $\mathscr{H}$:
\[
\mathscr{H}= \bigsqcup_{S\subset \N}  \mathscr{H}(\V, S).
\]
Moreover,  for any subset $S\subset \N$, the subset $
\{0\} \cup \mathscr{H}(\V, S)$
is a convex cone and the restriction of the metric projection
\[
P_{\mathcal{C}[[\V]]}\Big|_{\mathscr{H}(\V, S)}: \mathscr{H}(\V,S)\rightarrow \mathcal{C}[[\V]]
\]
is affine. That is, for any $\lambda_1, \lambda_2 >0$ and any $w_1, w_2 \in \mathscr{H}(\V, S)$, we have
\[
P_{\mathcal{C}[[\V]]} (\lambda_1 w_1 + \lambda_2 w_2) =  \lambda_1 P_{\mathcal{C}[[\V]]} (w_1) + \lambda_2 P_{\mathcal{C}[[\V]]} (w_2).
\]
\end{corollary}

By considering the analogue of Proposition \ref{prop-det-proj} for convex cone generated by finitely many vectors, we obtain in Corollary  \ref{prop-matrix} a result for positive definite matrices. This result  seems to be known in the litterature. We include it since we believe that our proof may be of its own interests.

\begin{corollary}\label{prop-matrix}
Assume that $A$ is a non-singular positive definite real-coefficient $n\times n$ matrix. Then for any $(c_1,\cdots,c_n) \in\R^n$, there exists a unique subset $S\subset\{1,2,\cdots, n\}$ and a unique  $x\in\R^S$ such that
\[
\left\{
  \begin{array}{ll}
x_i>0, & \hbox{for all $i\in S$;} 
\vspace{2mm}
\\
   \displaystyle \sum_{j\in S}a_{ij}x_j=c_i, & \hbox{for all $i\in S$;} \vspace{2mm}\\
    \displaystyle  \sum_{j\in S}a_{ij}x_j\geq c_i, & \hbox{for all $i\in\{1,2,\cdots, n\}$.}
  \end{array}
\right.
\]
\end{corollary}

\begin{remark*}
If $c_i\leq 0$ for all $1\leq i\leq n$, then we take $S=\emptyset$ in Corollary \ref{prop-matrix}.
\end{remark*}

\subsection{Computation of the metric projections}\label{sec-scheme}
We shall see in \S \ref{sec-func-th} that  Proposition~\ref{prop-det-proj} may be used to compute explicitly the metric projection of a vector onto the closed convex cone generated by a sequence.  Note that our method is different from the one presented in \cite{algo-app}. 

The general scheme is given as follows (note that although we focus on the case of an infinite sequence $\V= \{v_n\}_{n = 0}^\infty$, the same scheme is clearly still valid when the sequence is finite, that is $\V = \{v_n\}_{n = 0}^N$ with $N\in\N$). The main assumptions  for Proposition \ref{prop-det-proj} are
\begin{itemize}
\item[(i)] The sequence $\V= \{v_n\}_{n = 0}^\infty$ has no positive relations.
\item[(ii)] The closed convex cone $\overline{\mathcal{C}}[\V]$ generated by the sequence $\V$ is given by $\overline{\mathcal{C}}[\V] = \mathcal{C}[[\V]]$.
\end{itemize}
Under the above assumptions, assume that $w\in \mathscr{H}$ is a given vector and we want to compute the metric projection $P_{\mathcal{C}[[\V]]} (w)$. By Proposition \ref{prop-det-proj}, we shall and only need to determine a unique subset $S\subset \N$ (we will denote this subset by $S(w; \V, \mathscr{H})$ if it is necessary) and a unique sequence $(a_n)_{n\in S}$ with $a_n>0$ for all $n\in S$ such that \eqref{unique-S-a} is satisfied.   For further reference, let us denote
\begin{align}\label{Gamma-w-V}
\Gamma(w; \V, \mathscr{H}): = \Big\{v_n\Big| n \in S(w; \V, \mathscr{H}) \Big\}.
\end{align}
The main difficulty in computing the metric projection $P_{\mathcal{C}[[\V]]}(w)$ is to determine the unique subset $S$, or equivalently, to determine the set $\Gamma(w; \V, \mathscr{H})$. In general, it is not known to the authors whether there is an efficient way for determining such subset $S
\subset \N$ for an arbitrarily given vector $w\in \mathscr{H}$. However, for a given vector $w\in \mathscr{H}$ and a given subset $S\subset \N$, by  Proposition \ref{prop-det-proj}, it is relatively easier to determine whether the equality 
\[
\Gamma(w; \V, \mathscr{H})= \Big\{v_n\Big| n \in S \Big\}
\] holds or not (this is equivalent to determine whether $w\in \mathscr{H}(\V, S)$ or not,  where $\mathscr{H}(\V, S)$ is defined as in \eqref{def-H-S}). Let us explain how to do so when $S\subset \N$ is a given {\it finite} subset. Set
\[
M : = \Big(\langle v_m, v_n\rangle \Big)_{m, n \in \N}
\]
and let $M_S$ be the sub-matrix indexed by $S\times S$:
\[
M_S: = \Big(\langle v_m, v_n\rangle \Big)_{m, n \in S}.
\]
Then we need to solve the linear equation
\begin{align}\label{M-S-xy}
M_S x = y,
\end{align}
where  $x = (x_n)_{n\in S}\in \R^{S}$ is a column vector to be determined  and $y\in \R^S$ is the column vector defined by $y = (\langle w, v_n\rangle)_{n\in S}$. The assumption that the sequence $\V$ has no positive relations implies that the matrix  $M_S$ is non-singular (here we use the assumption that $S$ is finite) and thus the linear equation \eqref{M-S-xy} has a unique solution, denoted by $\widehat{x}\in \R^S$.  Now it remains to check whether the following conditions are satisfied:
\begin{align}\label{x-hat}
\left\{
\begin{array}{cl}
\widehat{x}_k>0 &\text{for all $k\in S$}
\vspace{2mm}
\\
\displaystyle \sum_{k\in S} \widehat{x}_k\langle v_k, v_n\rangle \ge \langle w, v_n\rangle & \text{for all $n\in \N \setminus S$}
\end{array}
\right..
\end{align}
If \eqref{x-hat} is satisfied, then $w\in \mathscr{H}(\V, S)$ and moreover, we obtain the desired metric projection $P_{\mathcal{C}[[\V]]}(w)$:
\[
P_{\mathcal{C}[[\V]]}(w)= \sum_{k\in S} \widehat{x}_k v_k.
\]
Otherwise, \eqref{x-hat} is not satisfied, then $w\notin \mathscr{H}(\V, S)$ and we shall try other subsets $S\subset\N$ for computing $P_{\mathcal{C}[[\V]]}(w)$.

\subsection{Applications in function theory}\label{sec-func-th}  Consider the Lebesgue measure on $[-1,1]$ and the associated Hilbert space $L^2([-1, 1])$. For easing the notation, set
\[
\mathcal{A}_{+} : = \mathcal{C}_{L^2([-1,1])}[[\{t^n\}_{n = 0}^\infty]],
\]
that is,
\[
\mathcal{A}_{+}:  = \bigg\{\sum_{n=0}^\infty a_nt^n\Big|a_n\geq 0,\text{ the series }\sum_{n=0}^\infty a_nt^n\text{ converges in }L^2([-1,1])\bigg\}.
\]
By Corollary \ref{cor-real-line}, the set $\mathcal{A}_{+}$
is a closed convex cone in $L^2([-1,1])$.  As before, the associated metric projection is denoted by  $P_{\mathcal{A}_{+}}: L^2([-1,1]) \rightarrow \mathcal{A}_{+}$.

\subsubsection{Power functions}

For any $\beta \in (0, \infty)$, let $h_\beta \in L^2([-1,1])$ be the power function defined by
\[
h_\beta(t) = |t|^{\beta}, \quad t\in [-1,1].
\]
 It turns out that the best approximation of $h_\beta$ by elements in $\mathcal{A}_{+}$ is given by a linear combination  of two elements $t^{2m}, t^{2m+2}$ with $2m$ and $2m+2$ the closest two even numbers  to $\beta$. For this reason, in what follows,  it is convenient for us to write
\[
\beta = 2m + \alpha \quad \text{with $m \in \N$ and $\alpha \in [0, 2)$.}
\]

\begin{proposition}\label{prop-h-alpha}
Let $\alpha\in [0,2)$ and $m\in \mathbb{N}$. Then
\begin{align}\label{eq-met-proj}
P_{\mathcal{A}_{+}}(h_{2m+\alpha})=a_m t^{2m}+b_m t^{2m+2},
\end{align}
where $a_m$ and $b_m$ are given by
\begin{align}\label{a-b-m}
\left\{
  \begin{array}{ll}
    \displaystyle a_m= a_m(\alpha): = \frac{(4m+1)(4m+3)}{(4m+1+\alpha)(4m+3+\alpha)}\frac{2-\alpha}{2},  \vspace{2mm} \\
   \displaystyle  b_m=b_m(\alpha):  = \frac{(4m+3)(4m+5)}{(4m+1+\alpha)(4m+3+\alpha)}\frac{\alpha}{2}.
  \end{array}
\right.
\end{align}
Moreover, the distance $d(h_{2m+\alpha}, \mathcal{A}_{+})$ is given by
\begin{align}\label{best-dist}
d(h_{2m+\alpha}, \mathcal{A}_{+}) =   \frac{ \sqrt{2} \alpha (2 - \alpha)}{(4m +\alpha +1)(4m +\alpha+3) \sqrt{4m+2\alpha +1}}
\end{align}
and the relative distance $\lambda(h_{2m+\alpha}, \mathcal{A}_{+})$ is given by
\[
\lambda(h_{2m+\alpha}, \mathcal{A}_{+}) =   \frac{\alpha (2 - \alpha)}{(4m +\alpha +1)(4m +\alpha+3)}.
\]
\end{proposition}

\begin{remark*}
Let $\alpha \in (0, 2)$. Then the coefficients $a_m(\alpha)$ and $b_m(\alpha)$ in Proposition \ref{prop-det-proj} satisfy
\[
a_m(\alpha) + b_m(\alpha) = \frac{(4m +3)(4m + 2\alpha +1)}{(4m+1+\alpha)(4m+3+\alpha)}>1.
\]
Hence the best approximation in $\mathcal{A}_{+}$ of the function $h_{2m+\alpha}$  is not a convex combination of $t^{2m}$ and $t^{2m+2}$.  In Figure \ref{3-best-approx}, we draw the graphs in the first quadrant the best approximation in $\mathcal{A}_{+}$ of $h_3$:
\[
h_3(t)= |t|^3, \quad P_{\mathcal{A}_{+}}(h_3)= \frac{35}{96} t^2 + \frac{21}{32} t^4.
\]
\end{remark*}

\begin{figure}[htbp!]
  \centering
  % Requires \usepackage{graphicx}
  \includegraphics[width=5in]{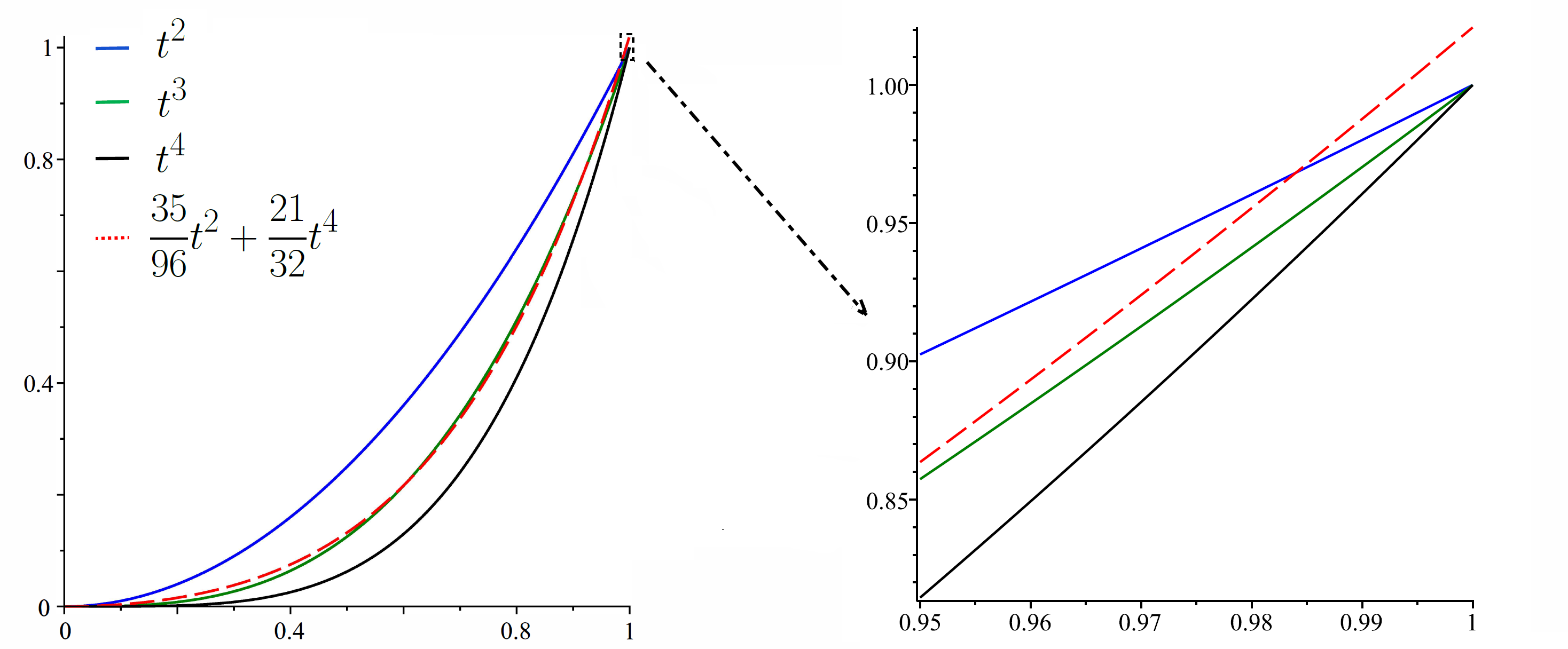}\\
  \caption{The best approximation of $|t|^3$ by polynomials with non-negative coefficients: the graphs in the first quadrant.} \label{3-best-approx}
\end{figure}

\begin{remark*}
Let $\alpha \in (0, 2)$. Recall the definition \eqref{Gamma-w-V}.  Proposition \ref{prop-det-proj} implies the equality:
\begin{align}\label{eq-gamma-t}
\Gamma\Big(|t|^{2m + \alpha}; \{t^n\}_{n = 0}^\infty, L^2([-1, 1])\Big) = \{t^{2m}, t^{4m}\}.
\end{align}
As we shall see in the proof of Proposition \ref{prop-det-proj}, it requires substantial efforts to prove the equality \eqref{eq-gamma-t}.

\begin{figure}[htbp!]
  \centering
  \includegraphics[width=2in]{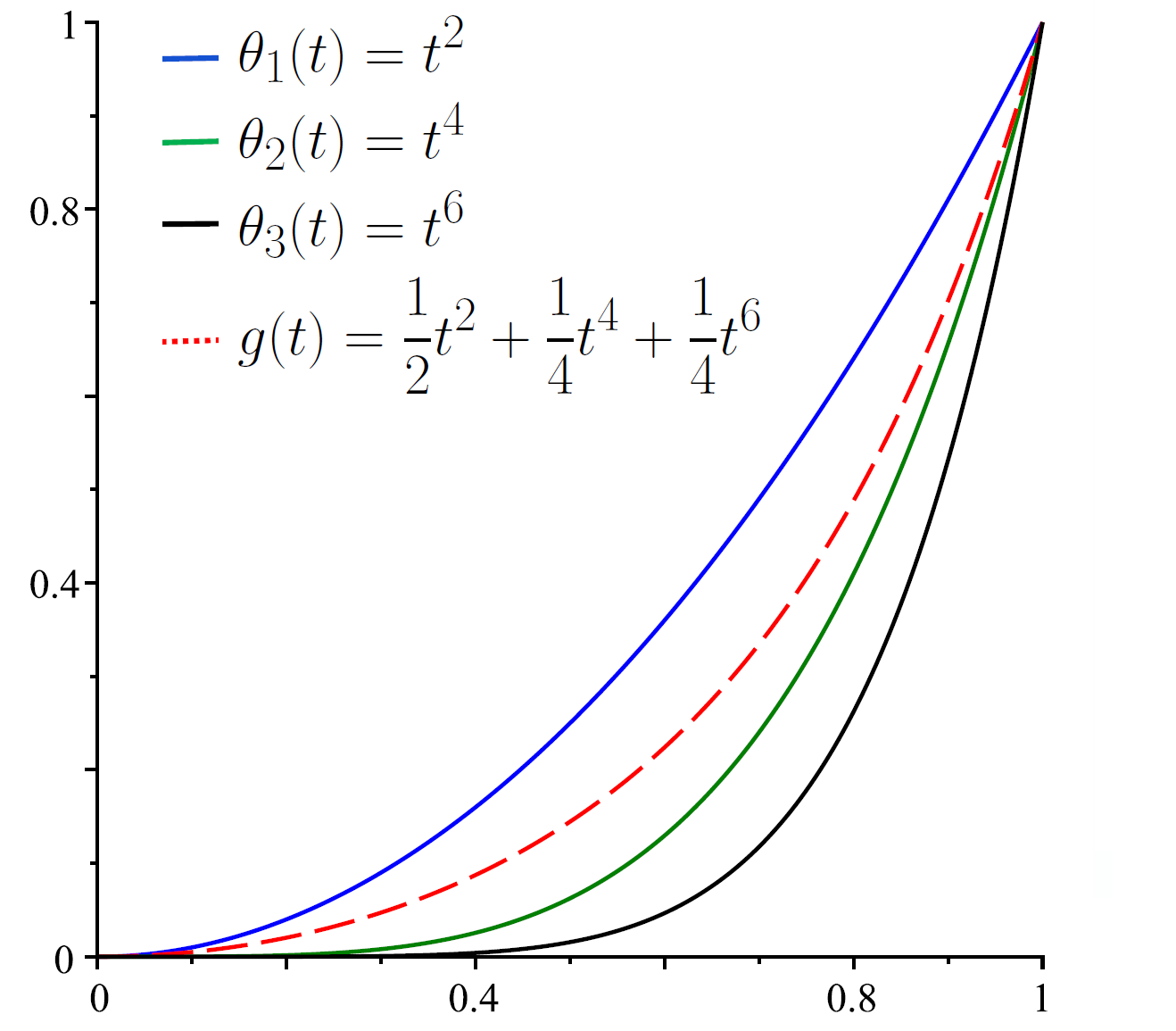}\\
  \caption{The best approximation in $\mathcal{C}(\theta_1, \theta_2, \theta_3)$ of the function $g$ is not given by positive combination of the two functions $\theta_1, \theta_2$, but is given by positive combination of all three functions $\theta_1, \theta_2, \theta_3$.}\label{fig-gh}
\end{figure}

For explaining clearer the subtlety of the equality \eqref{eq-gamma-t}, let us consider the Hilbert space $L^2([0,1])$.  It is easy to see that the equality \eqref{eq-gamma-t} is equivalent to the equality
\begin{align}\label{eq-gamma-unit}
\Gamma\Big(t^{2m + \alpha}; \{t^{2n}\}_{n = 0}^\infty, L^2([0, 1])\Big) = \{t^{2m}, t^{4m}\}.
\end{align}
Note  that the sequence $\{t^n\}_{n = 0}^\infty$ now is repalced by the sequence $\{t^{2n}\}_{n = 0}^\infty$.
That is, in the Hilbert space $L^2([0,1])$, the best approximation of the function $t^{2m+\alpha}$  by elements in the closed convex cone generated by $\{t^{2n}\}_{n=0}^\infty$  is given by a positive combination of the two functions $t^{2m}$ and $t^{2m+2}$.    One may  think that the equality \eqref{eq-gamma-unit} is a consequence of the following observation: among the graphs of all functions $t^0, t^2, t^4, \cdots$ on $[0,1]$,  the closest ones to that   of the function $t^{2m+\alpha}$ are exactly those of $t^{2m}$ and $t^{2m+2}$.  However, let us point out that  such kind of  observation in general is not sufficient for deriving the  equality \eqref{eq-gamma-unit}.  For instance,  as shown in   Figure \ref{fig-gh},  among the three graphs of $\theta_1, \theta_2, \theta_3$, the graphs of $\theta_1$ and $\theta_2$ are the closest ones to the graph of $g$. However, $g$ can be exactly approximated by elements in the convex cone $\mathcal{C}(\theta_1, \theta_2, \theta_3)$ generated by $\theta_1, \theta_2, \theta_3$:
\[
g= \frac{1}{2}\theta_1 + \frac{1}{4} \theta_2 + \frac{1}{4} \theta_3 \neq \lambda_1  \theta_1  + \lambda_2 \theta_2 \quad \text{for any $\lambda_1\ge 0, \lambda_2\ge 0$,}
\]
hence we have
\[
\Gamma \Big(g; \{\theta_1, \theta_2, \theta_3\},  L^2([0,1])\Big) = \{\theta_1, \theta_2, \theta_3\} \neq \{\theta_1, \theta_2\}.
\]
\end{remark*}

\begin{corollary}\label{cor-int}
For any $m\in \N$ and any positive Radon  measure $\nu$ on $[0, 2]$,  we have
\[
P_{\mathcal{A}_{+}} \left(\int_{[0,2]} h_{2m+\alpha} d\nu(\alpha)\right) = t^{2m}\int_{[0,2]} a_m(\alpha) d\nu(\alpha) + t^{2m+2} \int_{[0,2]} b_m(\alpha) d\nu(\alpha),
\]
where $a_m(\alpha), b_m(\alpha)$ are defined in \eqref{a-b-m}.
\end{corollary}
\begin{remark*}
The functions in Corollary \ref{cor-int} can be quite complicated: for instance, take $\nu$ the Lebesgue measure on $[0,\alpha]$ with $0<\alpha\le 2$, then we have
\[
\int_{[0,\alpha]} h_{2m+ \alpha'}(t) d\alpha'=  \frac{|t|^{2m+\alpha}- t^{2m}}{\log |t|}.
\]
\end{remark*}

\subsubsection{Signed power functions}

For any $\gamma \in (0, \infty)$, let $f_\gamma \in L^2([-1,1])$ be the signed power function defined by
\[
f_\gamma(t) = \sgn(t) |t|^{\gamma}, \quad t\in [-1,1].
\]

\begin{proposition}\label{prop-f-alpha}
Let $\alpha\in [0,2)$ and $m\in \mathbb{N}$. Then
\begin{align}\label{eq-sgn-proj}
P_{\mathcal{A}_{+}}(f_{2m+1+\alpha})=c_m t^{2m+1}+ d_m t^{2m+3},
\end{align}
where $c_m$ and $d_m$ are given by
\begin{align}\label{c-d-m}
\left\{
  \begin{array}{ll}
    \displaystyle c_m= c_m(\alpha): = \frac{(4m+3)(4m+5)}{(4m+3+\alpha)(4m+5+\alpha)} \frac{2-\alpha}{2},  \vspace{2mm} \\
   \displaystyle  d_m=d_m(\alpha):  = \frac{(4m+5)(4m+7)}{(4m+3+\alpha)(4m+5+\alpha)} \frac{\alpha}{2}.
  \end{array}
\right.
\end{align}
Moreover, the distance $d(f_{2m+1+\alpha}, \mathcal{A}_{+})$ is given by 
\begin{align}\label{best-dist-sgn}
d(f_{2m+1+\alpha}, \mathcal{A}_{+})=     \frac{ \sqrt{2} \alpha (2 - \alpha)}{(4m +\alpha +3)(4m +\alpha+5) \sqrt{4m+2\alpha +3}}
\end{align}
and the relative distance $\lambda(f_{2m+1+\alpha}, \mathcal{A}_{+})$ is given by
\[
\lambda(f_{2m+1+\alpha}, \mathcal{A}_{+}) =   \frac{\alpha (2 - \alpha)}{(4m +\alpha +3)(4m +\alpha+5)}.
\]
\end{proposition}

\subsubsection{Indicator functions and non-decreasing functions}
For any $a\in [-1, 1)$, let $\psi_a$ be the function defined by
\[
\psi_a(t) = \mathds{1}(t\ge a) = \left\{
\begin{array}{cl}
0 & \text{if $t\in [-1, a)$}
\vspace{2mm}
\\
1&  \text{if $t \in [a, 1]$}
\end{array}
\right..
\]

Recall the definition \eqref{Gamma-w-V}.
\begin{proposition}\label{prop-indicator}
Assume that $a\in [-1,1)$. Then the equality
\begin{align}\label{gamma-set-3}
\Gamma\Big(\psi_a; \{t^n\}_{n = 0}^\infty, L^2([-1,1])\Big) = \{t^n| n=0,  1, 2\}
\end{align}
holds if and only if 
\[
0< a \le \frac{1}{\sqrt{5}}.
\]  Moreover, for any $a\in (0, 1/\sqrt{5}]$, the metric projection $P_{\mathcal{A}_{+}} (\psi_a)$ is given by
\begin{align}\label{P-C-psi-a}
P_{\mathcal{A}_{+}} (\psi_a)= \frac{1}{8} (4-9a+5a^3) + \frac{3}{4}(1 - a^2) t + \frac{15}{8} (a-a^3) t^2.
\end{align}
\end{proposition}

\begin{proposition}\label{prop-neg-a}
Assume that $a\in [-1,1)$.  Then the equality
\begin{align}\label{gamma-set-2}
\Gamma\Big(\psi_a; \{t^n\}_{n = 0}^\infty, L^2([-1,1])\Big) = \{t^n| n= 0, 1\}
\end{align}
holds if and only if 
\[
-\frac{1}{\sqrt{5}}\le a \le 0.
\]  Moreover, for any $a\in [-1/\sqrt{5}, 0]$, the metric projection $P_{\mathcal{A}_{+}} (\psi_a)$ is given by
\begin{align}\label{P-neg-a}
P_{\mathcal{A}_{+}} (\psi_a)= \frac{1}{2} (1-a) + \frac{3}{4}(1 - a^2) t.
\end{align}
\end{proposition}

\begin{proposition}\label{prop-big-a}
Assume that $a\in [-1, 1)$. Then the equality
\begin{align}\label{gamma-set-4}
\Gamma\Big(\psi_a; \{t^n\}_{n = 0}^\infty, L^2([-1,1])\Big) = \{t^n| n= 0, 1,2,3\}
\end{align}
holds if and only if
\[
\frac{1}{\sqrt{5}} < a < \frac{\sqrt{105} - 5}{10}.
\]
Moreover, for any $a\in (\frac{1}{\sqrt{5}}, \frac{\sqrt{105} - 5}{10})$, the metric projection $P_{\mathcal{A}_{+}} (\psi_a)$ is given by
\begin{align}\label{P-big-a}
\begin{split}
P_{\mathcal{A}_{+}} (\psi_a)= &  \frac{1}{8} (1 -a)(4-5a-5a^2)  +  \frac{15}{32} (1 - a^2) (3-7a^2) t
\\
 & \quad + \frac{15}{8}a (1-a^2)  t^2 + \frac{35}{32}(1 -a^2)(5a^2 -1) t^3.
\end{split}
\end{align}
\end{proposition}

\begin{remark*}
Propositions \ref{prop-indicator} and \ref{prop-big-a} may lead  one to guess that there exists a sequence of critical points $\{b_k\}_{k=0}^\infty$ with $0< b_0<b_1<\cdots < b_k < \cdots <1$ such that 
\[
\Gamma\Big(\psi_a; \{t^n\}_{n = 0}^\infty, L^2([-1,1])\Big) = \{ t^n| n = 0, 1, 2, \cdots, k+1, k+2\} \quad \text{for all $a\in (b_{k-1}, b_k)$.}
\]
This is however not clear to the authors at the time of writing. Indeed, the situation becomes more involved when $a$ is close to $1$.  On the other hand, for negative $a$, the situation seems to be different. By \cite{Wulbert-1968}, since $\mathcal{A}_{+}$ is a closed convex cone in a Hilbert space, the metric projection $P_{A_{+}}: L^2([-1,1])\rightarrow \mathcal{A}_{+}$ is a continuous map.   Therefore, by the formula \eqref{gamma-set-2}, there exists $\varepsilon>0$ such that
\[
\{0, 1\} \subset  \Gamma\Big(\psi_a; \{t^n\}_{n = 0}^\infty, L^2([-1,1])\Big) \quad \text{for all $a\in (-\frac{1}{\sqrt{5}}-\varepsilon, -\frac{1}{\sqrt{5}})$}.
\]
But by Propositions \ref{prop-indicator}, \ref{prop-neg-a} and \ref{prop-big-a}, for any $a\in (-\frac{1}{\sqrt{5}}-\varepsilon, -\frac{1}{\sqrt{5}})$, we know that the set $\Gamma\Big(\psi_a; \{t^n\}_{n = 0}^\infty, L^2([-1,1])\Big)$ can not be any one of the three sets: $\{0, 1\}$,  $\{0, 1, 2\}$ or $\{0, 1, 2, 3\}$. 
\end{remark*}

Propositions \ref{prop-indicator}, \ref{prop-neg-a}  and \ref{prop-big-a} allow us to compute the metric projections onto the closed convex cone $\mathcal{A}_{+}$ for functions in three large classes respectively. Let us state  the consequence of Proposition \ref{prop-indicator} in Corollary \ref{cor-monotone} below, the consequences of Propositions~\ref{prop-neg-a} and \ref{prop-big-a} are similar and will be omitted.

Let $\mathcal{M}_{[0, 1/\sqrt{5}]}$ denote the class of functions on $[-1,1]$ consisting of all non-decreasing {\it right continuous} non-negative functions $\varphi: [-1, 1]\rightarrow \R_{+}$ such that
\[
\varphi|_{[-1, 0)} \equiv 0 \an  \varphi|_{[1/\sqrt{5}, 1]}\equiv \text{constant}.
\]
Note that  any $\varphi\in \mathcal{M}_{[0, 1/\sqrt{5}]}$ uniquely determines a non-negative Radon measure,  denoted by $d\varphi$, on the interval $[-1,1]$,  by the formula
\[
d\varphi( [-1, t]) : = \varphi(t) \quad \text{for any $t \in [-1,1]$}.
\]
Moreover,  the support $\supp(d\varphi)$  of the Radon measure $d\varphi$ satisfies $\supp(d\varphi) \subset [0, 1/\sqrt{5}]$ and we have
\begin{align*}
\varphi (t)= \int_{[0,1/\sqrt{5}]} \psi_a (t) d\varphi(a) = \int_{[0,1/\sqrt{5}]}  \mathds{1}(t\ge a) d\varphi(a), \quad t\in [-1,1].
\end{align*}
 On the other hand, for any Radon measure on $[-1,1]$ with support $\supp(\nu)\subset [0, 1/\sqrt{5}]$, the function $\varphi_\nu$, defined by  the formula
\[
\varphi_\nu(t): = \nu([-1, t]) \quad \text{for all $t \in [-1,1]$},
\]
belongs to the class $\mathcal{M}_{[0, 1/\sqrt{5}]}$.

\begin{corollary}\label{cor-monotone}
For any function $\varphi \in \mathcal{M}_{[0, 1/\sqrt{5}]}$, we have
\begin{align*}
P_{\mathcal{A}_{+}} (\varphi) = A_\varphi  + B_\varphi t + C_\varphi t^2,
\end{align*}
with
\begin{align*}
\left\{
\begin{array}{l}
\displaystyle A_\varphi =     \int_{[0, 1/\sqrt{5}]}  \frac{1}{8} (4-9a+5a^3)  d\varphi(a)
\vspace{2mm}
\\
\displaystyle B_\varphi=  \int_{[0, 1/\sqrt{5}]} \frac{3}{4} (1 - a^2)d\varphi(a)
\vspace{2mm}
\\
\displaystyle C_\varphi =    \int_{[0, 1/\sqrt{5}]}  \frac{15}{8} (a-a^3)  d\varphi(a)
\end{array}
\right..
\end{align*}
\end{corollary}

\section{Closedness of convex cones}
In this section, we prove Theorem \ref{prop-closed} and Corollary \ref{cor-real-line}.

\begin{proof}[Proof of Theorem \ref{prop-closed}]
Let $\mathscr{H}_\V \subset \mathscr{H}$ denote the closed linear span of vectors in $\V$:
\[
\mathscr{H}_\V : = \overline{\spann} (\V) \subset \mathscr{H}.
\] 
Let $\mathscr{H}(\mu, \mathcal{W})$ denote the closed linear span of the sequence $\{(t^n, w_n)\}_{n=0}^\infty$ in  the real Hilbert space $L^2(\mu) \oplus \mathscr{K}$:
\[
\mathscr{H}(\mu, \mathcal{W}):=  \overline{\spann}\Big\{(t^n, w_n)\Big| n\in \N\Big\} \subset L^2(\mu) \oplus \mathscr{K}.
\]
The equalities \eqref{v-n-inner} imply that the map $(t^n, w_n) \mapsto \lambda_n v_n$ can be extended to a linear isometric isomorphism between $\mathscr{H}(\mu, \mathcal{W})$ and $\mathscr{H}_\V$. To complete the proof of Theorem~\ref{prop-closed}, we shall prove that the following set
\[
\mathcal{C}(\mu, \mathcal{W}): =\bigg\{\sum_{n=0}^\infty a_n(t^n, w_n) \Big|a_n\geq 0,\text{\,the series $\sum_{n=0}^\infty a_n (t^n, w_n)$ converges in $L^2(\mu) \oplus \mathscr{K}$}\bigg\}
\]
is closed in $L^2(\mu) \oplus \mathscr{K}$.

Note first that for any element $(f, u) \in \mathcal{C}(\mu, \mathcal{W}) \subset L^2(\mu) \oplus \mathscr{K}$ with
\begin{align}\label{L-2-eq}
f = \sum_{n = 0}^\infty a_n t^n,  \quad a_n \ge 0,
\end{align}
where the series converges in $L^2(\mu)$ and the equality is understood as elements in $L^2(\mu)$,  there exists a subsequence $\{N_k\}_{k=0}^\infty$ of positive integers with $0< N_0< N_1 < \cdots$ such that 
\begin{align}\label{pt-limit}
 \lim_{k\to\infty} \sum_{n=0}^{N_k} a_n t^n  = f(t)<\infty \quad \text{for $\mu$-a.e. $t\in [0, s_\mu)$}.
\end{align}
By the assumption $a_n\ge 0$ for all $n\in \N$,  for any $t_0\in [0, s_\mu)$ such that the limit equality \eqref{pt-limit} holds, we have 
\[
\lim_{k\to\infty} \sum_{n=0}^{N_k} a_n t_0^n = \sum_{n=0}^\infty a_n t_0^n: = \lim_{N\to\infty} \sum_{n=0}^N a_n t_0^n. 
\]
It follows that we have the following $\mu$-almost everywhere equality:
\begin{align}\label{pt-eq}
f(t)  = \sum_{n = 0}^\infty a_n t^n <\infty \quad  \text{for $\mu$-a.e. $t\in [0, s_\mu)$}.
\end{align}
By the definition of $s_\mu$ and the assumption $a_n\ge 0$ for all $n\in\N$, we even have
\begin{align*}
\sum_{n = 0}^\infty a_n t^n<\infty \quad \text{ for all $t\in [0, s_\mu)$}.
\end{align*}

Now let $\{(f_k, u_k)\}_{k = 0}^\infty$ be a sequence in $\mathcal{C}(\mu, \mathcal{W})$:
\[
(f_k, u_k)=\Big(\sum_{n=0}^\infty a_n^{(k)}t^n, \sum_{n=0}^\infty a_n^{(k)}w_n\Big)\in \mathcal{C}(\mu, \mathcal{W}),
\]
and assume that $(f_\infty, u_\infty) \in L^2(\mu) \oplus \mathscr{K}$ is the limit of the sequence $\{(f_k, u_k)\}_{k = 0}^\infty$:
\begin{equation}\label{E:L2conv-pos}
(f_k, u_k)\xrightarrow[k\to\infty]{\text{in $L^2(\mu) \oplus \mathscr{K}$}} (f_\infty, u_\infty).
\end{equation}
We want to show that $(f_\infty, u_\infty) \in \mathcal{C}(\mu, \mathcal{W})$. Indeed, \eqref{E:L2conv-pos} implies, up to passing to a subsequence if necessary, that
\begin{align}\label{pt-cv-seq}
f_k(t)  = \sum_{n=0}^\infty a_n^{(k)}t^n\xrightarrow{k\to\infty} f_\infty(t)\quad \text{for $\mu$-a.e. $t\in \R_{+}$}.
\end{align}
Note that the condition  $\mu\big(\R_{+} \setminus [0, s_\mu)\big) = 0$ implies in particular that the support $\supp(\mu)$ is an infinite subset of $\R_{+}$. By \eqref{pt-cv-seq}, there exists a sequence $\tau_0< \tau_1< \cdots$ in $[0, s_\mu)$ such that
\[
\lim_{m\to\infty} \tau_m = s_\mu
\]
and 
\[
f_k(\tau_m) = \sum_{n=0}^\infty a_n^{(k)}\tau_m^n \xrightarrow{k\to\infty} f_\infty(\tau_m) \in [0, \infty) \quad \text{for all $m\in \N$}.
\]
It follows that
\[
M_m: = \sup_{k} \sum_{n=0}^\infty a_n^{(k)}\tau_m^n \in [0, \infty) \quad \text{for all $m\in \N$}.
\]
Since all coefficients $a_n^{(k)} \ge 0$, we have
\begin{align}\label{M-bdd}
0\le a_n^{(k)} \tau_m^n \le M_m \quad \text{for all $m, k \in\N$}.
\end{align}
Using the compactness of $[0, M_m]$ and the canonical Cantor's diagonal method,  we may extract a subsequence of positive integers, denoted by  $0< k_1<k_2<k_3<\cdots$, such that  for all $m, n\in \N$, the following limits exist:
\[
\lim_{i\to\infty} a_n^{(k_i)} \tau_m^n \in [0, M_m].
\]
But this means that the limits $\lim_{i\to\infty} a_n^{(k_i)}$ exist for all $n\in \N$ and moreover,
\begin{align}\label{def-a-inf}
a_n^{(\infty)}: =\lim_{i\to\infty} a_n^{(k_i)} \in [0, \frac{M_m}{\tau_m^n}]  \quad \text{for all $m, n \in \N$}.
\end{align}
Now for any $t\in [0, \tau_{m-1}]$, we have
\[
0\le \frac{t}{\tau_{m}}\le \frac{\tau_{m-1}}{\tau_{m}}<1
\]
and hence by \eqref{M-bdd},
\[
\sum_{n=0}^\infty \sup_i |a_n^{(k_i)} t^n | \le \sum_{n = 0}^\infty  \frac{M_m}{\tau_m^n}  t^n \le \sum_{n=0}^\infty M_m \left(\frac{\tau_{m-1}}{\tau_m}\right)^n<\infty \quad \text{for all $t\in [0, \tau_{m-1}]$}.
\]
Therefore, by the Dominated Convergence Theorem,
\begin{align}\label{DCT}
\lim_{i\to\infty}\sum_{n=0}^\infty a_n^{(k_i)} t^n  = \sum_{n=0}^\infty \lim_{i\to\infty} a_n^{(k_i)} t^n   =  \sum_{n=0}^\infty a_n^{(\infty)} t^n  \quad \text{for all $t\in [0, \tau_{m-1}]$.}
\end{align}
Combining \eqref{pt-eq}, \eqref{pt-cv-seq} and \eqref{DCT}, we obtain
\[
f_\infty(t) = \sum_{n=0}^\infty a_n^{(\infty)} t^n  \quad \text{for $\mu_{m-1}$-a.e. $t\in [0, \tau_{m-1}]$},
\]
where $\mu_{m-1} = \mu|_{[0, \tau_{m-1}]}$ is the restiction of the measure $\mu$ on $[0, \tau_{m-1}]$. Since $m$ is arbitrary, we have
\begin{align}\label{lim-pt-eq}
f_\infty(t)  = \sum_{n= 0}^\infty a_n^{(\infty)} t^n  \quad \text{for $\mu$-a.e. $t\in [0, s_\mu)$}.
\end{align}
 Combining \eqref{lim-pt-eq} with the condition  $\mu\big(\R_{+} \setminus [0, s_\mu)\big) = 0$, we obtain 
\begin{align}\label{lim-ae-eq}
f_\infty(t)  = \sum_{n= 0}^\infty a_n^{(\infty)} t^n  \quad \text{for $\mu$-a.e. $t\in \R_{+}$}.
\end{align}

We then need to show that the $\mu$-almost everywhere equality \eqref{lim-ae-eq} implies the following $L^2$-norm convergence:
\begin{align}\label{f-inf-L2}
\lim_{N\to\infty}\Big\|\sum_{n=0}^N a_n^{(\infty)} t^n - f_\infty\Big\|_{L^2(\mu)} = 0.
\end{align}
But this again follows from the Dominated Convergence Theorem. Indeed, since   $a_n^{(\infty)} \ge 0$  we have
\[
\sup_{N}\Big(\sum_{n=N+1}^\infty a_n^{(\infty)} t^n\Big)^2  \le \Big(\sum_{n=0}^\infty a_n^{(\infty)} t^n\Big)^2  = f_\infty(t)^2.
\]
The above inequality combined with  the assumption $f_\infty\in L^2(\R_{+}, \mu)$ implies
\[
\lim_{N\to\infty}\Big\|\sum_{n=0}^N a_n^{(\infty)} t^n - f_\infty\Big\|_{L^2(\mu)}^2  = \lim_{N\to\infty}  \int_{\R_{+}} \Big(\sum_{n=N+1}^\infty a_n^{(\infty)} t^n\Big)^2d\mu(t) = 0.
\]

Finally, it remains to show that
\begin{align}\label{u-inf-lim}
 \lim_{N\to\infty}\Big\|\sum_{n=0}^N a_n^{(\infty)} w_n - u_\infty \Big\| = 0.
\end{align}
By the assumption \eqref{ab-cond} and the fact that $a_n^{(\infty)} \ge 0$ for all $n\in \N$, we have 
\[
\Big\|\sum_{n= N}^M a_n^{(\infty)} w_n\Big\| \le C \Big\| \sum_{n= N}^M a_n^{(\infty)} t^n \Big\|_{L^2(\mu)} \text{for all $N, M\in\N$ with $N\le M$}.
\]
Thus  the convergence of the series $\sum_{n=0}^\infty a_n^{(\infty)} t^n$ in the space $L^2(\mu)$ implies the convergence of the series $\sum_{n=0}^\infty a_n^{(\infty)} w_n$ in $\mathscr{K}$. 
  By \eqref{E:L2conv-pos}, we have
\begin{align}\label{u-inf}
u_\infty = \lim_{k\to\infty} \sum_{n=0}^\infty a_n^{(k)} w_n.
\end{align}
Let $\{k_i\}_{i}$ be the subsequence of positive integers chosen as above.
For any $N\in \N$ and any $i\in \N$, by the assumption \eqref{ab-cond} and the fact that $a_n^{(k_i)} \ge 0$, we have
\begin{align}\label{tr-ineq}
\begin{split}
 \Big\|\sum_{n=0}^N a_n^{(\infty)} w_n - u_\infty \Big\| \le &  \Big\| \sum_{n=0}^N a_n^{(\infty)} w_n -  \sum_{n=0}^N a_n^{(k_i)} w_n\Big\| + \Big\| \sum_{n=N+1}^\infty a_n^{(k_i)} w_n\Big\|
\\
& + \Big\|   \sum_{n=0}^\infty a_n^{(k_i)} w_n -u_\infty\Big\|
\\
\le&    \sum_{n=0}^N | a_n^{(\infty)} - a_n^{(k_i)}| \cdot \| w_n\| + C \Big\| \sum_{n=N+1}^\infty a_n^{(k_i)} t^n \Big\|_{L^2(\mu)}
\\
& + \Big\|   \sum_{n=0}^\infty a_n^{(k_i)} w_n -u_\infty\Big\|.
\end{split}
\end{align}
Combining \eqref{def-a-inf}, \eqref{u-inf} and \eqref{tr-ineq}, for any $N\in \N$, we have
\begin{align*}
\Big\|\sum_{n=0}^N a_n^{(\infty)} w_n - u_\infty \Big\|  \le C   \liminf_{i\to\infty}\Big\| \sum_{n=N+1}^\infty a_n^{(k_i)} t^n \Big\|_{L^2(\mu)}.
\end{align*}
Note also that \eqref{E:L2conv-pos}  and \eqref{f-inf-L2} together imply 
\[
\lim_{i\to\infty} \Big\| \sum_{n = 0}^\infty a_n^{(k_i)} t^n - \sum_{n = 0}^\infty a_n^{(\infty)} t^n \Big\|_{L^2(\mu)}= 0.
\]
Therefore, for any fixed $N\in \N$, we have
\[
\liminf_{i\to\infty}\Big\| \sum_{n=N+1}^\infty a_n^{(k_i)} t^n \Big\|_{L^2(\mu)} = \Big\| \sum_{n=N+1}^\infty a_n^{(\infty)} t^n \Big\|_{L^2(\mu)}
\]
and hence
\[
\Big\|\sum_{n=0}^N a_n^{(\infty)} w_n - u_\infty \Big\| \le C   \Big\| \sum_{n=N+1}^\infty a_n^{(\infty)} t^n \Big\|_{L^2(\mu)}.
\]
Thus we have
\[
\limsup_{N\to\infty} \Big\|\sum_{n=0}^N a_n^{(\infty)} w_n - u_\infty \Big\|  \le C  \limsup_{N\to\infty}  \Big\| \sum_{n=N+1}^\infty a_n^{(\infty)} t^n \Big\|_{L^2(\mu)} = 0.
\]
This completes the proof of the limit relation \eqref{u-inf-lim}. 
\end{proof}

\begin{remark*}
The implication \eqref{L-2-eq} $\Longrightarrow$ \eqref{pt-eq} relies heavily on the non-negativity of the functions $a_n t^n$ on $[0, s_\mu)$. In general, the equality  $f = \sum_{n = 0}^\infty f_n$ in the Hilbert space $L^2(\mu)$
does not imply the $\mu$-almost everywhere equality $f(t) \stackrel{\text{$\mu$-a.e.}}{=\joinrel=} \sum_{n= 0}^\infty f_n(t)$. 
\end{remark*}

\begin{proof}[Proof of Corollary \ref{cor-real-line}]
By writing $\mu= \nu|_{\R_{+}}$, we have 
\[
L^2(\R, \nu) = L^2(\R_{+}, \mu) \oplus L^2(\R_{-}, \nu).
\]
The assumption \eqref{neg-less-pos} implies that for any finitely supported sequence $\{a_n\}_{n=0}^\infty$ of non-negative numbers, we have
\begin{align*}
\Big\|\sum_{n=0}^\infty a_n t^n \Big\|_{L^2(\R_{-}, \nu)}^2 &  = \sum_{n, m \ge 0}  a_n a_m \int_{\R_{-}} t^{m + n} d\nu(t)
\\
&  \le \sum_{m,n \ge 0 \atop \text{$m+n$ is even}}  a_n a_m \int_{\R_{+}} t^{m+n} d\mu(t)
\\
& \le  C \sum_{m,n \ge 0}  a_n a_m \int_{\R_{+}} t^{m+n} d\mu(t) =  C\Big\|\sum_{n=0}^\infty a_n t^n \Big\|_{L^2(\R_{+}, \mu)}^2,
\end{align*}
where in the first inequality, we have used the inequalities
\[
\int_{\R_{-}} t^{2k+1} d\nu(t) \le 0 \quad \text{for all $k\in \N$}.
\]
Now Corollary \ref{cor-real-line} follows immediately from Theorem \ref{prop-closed}.
\end{proof}

\section{Characterization of the metric projection}

\subsection{Proof of Proposition \ref{prop-det-proj}}
Recall that $\mathcal{C}[[\V]]$ is assumed to be closed.  For any $w\in\mathscr{H}$, by a classical result on the metric projections onto a closed convex set (cf. \cite[Lemma 1.1]{Zarantonello-1971}), $P_{\mathcal{C}[[\V]]} (w)$ is uniquely determined by
\begin{equation}\label{E:projection-convex-set}
\left\{
  \begin{array}{ll}
    P_{\mathcal{C}[[\V]]} (w) \in\mathcal{C}[[\V]], &\vspace{2mm} \\
    \langle w-P_{\mathcal{C}[[\V]]} (w), u-P_{\mathcal{C}[[\V]]}(w)\rangle \leq 0, & \hbox{ for all $u\in\mathcal{C}[[\V]]$.}
  \end{array}
\right.
\end{equation}
By the assumption that $\mathcal{C}[[\V]]$ is a closed convex cone, we may use \cite[Lemma 3]{IM-1991} to obtain
\[
\langle w-P_{\mathcal{C}[[\V]]} (w),P_{\mathcal{C}[[\V]]} (w)\rangle=0.
\]
This combined with \eqref{E:projection-convex-set} implies that $P_{\mathcal{C}[[\V]]} (w)$ is uniquely determined by
\begin{align}\label{final-cond}
\left\{
  \begin{array}{ll}
    P_{\mathcal{C}[[\V]]} (w) \in\mathcal{C}[[\V]], &  \vspace{2mm}\\
    \langle w-P_{\mathcal{C}[[\V]]} (w) , P_{\mathcal{C}[[\V]]} (w)\rangle =0, & \vspace{2mm} \\
    \langle w-P_{\mathcal{C}[[\V]]} (w),u\rangle\leq 0, & \hbox{ for all $u\in\mathcal{C}[[\V]]$.}
  \end{array}
\right.
\end{align}
By the definition of $\mathcal{C}[[\V]]$, the condition
\[
\langle w-P_{\mathcal{C}[[\V]]}(w), u\rangle \leq 0 \quad \text{for any $u\in \mathcal{C}[[\V]]$}
\]
is satisfied if and only if
\begin{align}\label{inner-v-n}
\langle w-P_{\mathcal{C}[[\V]]}(w), v_n\rangle\leq 0 \quad  \text{for all $n\in \N$}.
\end{align}
By writing
\[
P_{\mathcal{C}[[\V]]}(w)=\sum_{n=0}^\infty a_nv_n \quad \text{with $a_n\ge 0$ for all $n\in \N$,}
\]
we  have
\begin{align}\label{0-sum}
0 =  \langle w-P_{\mathcal{C}[[\V]]} (w) , P_{\mathcal{C}[[\V]]} (w)\rangle =   \sum_{n =0}^\infty  a_n \langle w-P_{\mathcal{C}[[\V]]} (w) , v_n \rangle.
\end{align}
Combining \eqref{inner-v-n} and \eqref{0-sum}, we obtain
\begin{align}\label{pos-a-n}
\langle w-P_{\mathcal{C}[[\V]]} (w) , v_n \rangle  =0 \quad \text{for all those $n\in\N$ such that $a_n>0$.}
\end{align}
On the other hand, \eqref{pos-a-n} clearly implies  the equality \eqref{0-sum}.
Therefore, the condition \eqref{final-cond} is equivalent to
\begin{align*}
\left\{
  \begin{array}{ll}
\displaystyle    P_{\mathcal{C}[[\V]]}(w)=\sum_{n=0}^\infty a_nv_n\in \mathcal{C}[[\V]], &  \vspace{2mm}\\
    \displaystyle \sum_{k\in \N} a_k \langle v_k, v_n\rangle  =  \langle w, v_n\rangle, & \hbox{for all $n\in\mathbb{N}$ such that $a_n >0$,} \vspace{2mm}\\
\displaystyle   \sum_{k\in \N} a_k \langle v_k, v_n\rangle \ge \langle w, v_n\rangle, & \hbox{for all $n\in\mathbb{N}$.}
  \end{array}
\right.
\end{align*}
This completes the proof of Proposition \ref{prop-det-proj}.

\subsection{Proof of Corollary \ref{prop-matrix}}

If $\V \subset \mathscr{H}$ is a finite set, then the convex cone $\mathcal{C}[\V]$ generated by $\V$ is always closed, cf. e.g. \cite[p. 236]{Zarantonello-1971} and \cite[p. 25]{Borwein-2006}.

\begin{lemma}\label{L:existence-rv}
Let $X_1,\cdots, X_n$ be linear independent real-valued random variables, all of which are of finite second moment. Then for any $c=(c_1,\cdots,c_n)\in\mathbb{R}^n$, there exists a real-valued random variable $Y$ of finite second moment such that
\begin{equation}\label{E:covariance}
c_i=\mathbb{E}[YX_i],\quad 1\leq i\leq n.
\end{equation}
\end{lemma}
\begin{proof}
Set
\[
A=(\mathbb{E}(X_iX_j))_{1\leq i,j\leq n}=
\mathbb{E}\bigg[\left(
\begin{array}{c}
X_1 \\
\vdots \\
X_n \\
\end{array}
\right)
\left(X_1 \, \cdots\, X_n
\right)
\bigg].
\]
Since $X_1,\cdots, X_n$ are linear independent, by Schmidt orthogonalization method, there exists a non-singular matrix $P$ such that the random variables $Z_i$'s defined by
\[
\left(
  \begin{array}{c}
    Z_1 \\
    \vdots \\
    Z_n \\
  \end{array}
\right)
=P\left(
    \begin{array}{c}
      X_1 \\
      \vdots \\
      X_n \\
    \end{array}
  \right)
\]
are orthogonal and $\E(Z_i^2)>0$ for all $i\in\{1, \cdots, n\}$. Note that the condition \eqref{E:covariance} can be written as
\begin{equation}\label{E:newcovergence}
\left(
  \begin{array}{c}
    c_1 \\
    \vdots \\
    c_n \\
  \end{array}
\right)
=\mathbb{E}
\bigg[
\left(
  \begin{array}{c}
    X_1 \\
    \vdots \\
    X_n \\
  \end{array}
\right)Y
\bigg].
\end{equation}
Since $P$ is non-singular, (\refeq{E:newcovergence}) is equivalent to
\begin{equation}\label{E:newcovergence02}
\left(
  \begin{array}{c}
    c_1' \\
    \vdots \\
    c_n' \\
  \end{array}
\right): =
P\left(
  \begin{array}{c}
    c_1 \\
    \vdots \\
    c_n \\
  \end{array}
\right)
=P\mathbb{E}
\bigg[
\left(
  \begin{array}{c}
    X_1 \\
    \vdots \\
    X_n \\
  \end{array}
\right)Y
\bigg]
=\mathbb{E}
\bigg[
\left(
  \begin{array}{c}
    Z_1 \\
    \vdots \\
    Z_n \\
  \end{array}
\right)Y
\bigg].
\end{equation}
If we set
\[
Y=\sum_{i=1}^n \frac{c_i' Z_i}{\E(Z_i^2)},
\]
then $Y$ satisfies the condition (\refeq{E:newcovergence02}). This completes the whole proof.
\end{proof}

\begin{proof}[Proof of Corollary \ref{prop-matrix}]
Since $A$ is a non-singular positive definite matrix, there exist real-valued square-integrable and linear independent random variables $X_1,\cdots,X_n$ such that
\[
\mathbb{E}(X_i)=0 \an A=(\mathbb{E}(X_iX_j))_{1\leq i,j\leq n}.
\]
By Lemma \ref{L:existence-rv}, there exists a real-valued square-integrable random variable $Y$ such that
\[
c_i=\mathbb{E}(X_iY),\quad 1\leq i\leq n.
\]
The convex cone in the associated Hilbert space of square-integrable random variables generated by $X_1, \cdots, X_n$ is
\[
\mathcal{C}(X_1,\cdots,X_n)=\Big\{\sum_{i=1}^n b_iX_i\Big|b_i\geq 0\Big\}.
\]
Since $\mathcal{C}(X_1,\cdots,X_n)$ is closed (cf. \cite[p. 236]{Zarantonello-1971} and \cite[p. 25]{Borwein-2006}),  there exists a unique  $Z\in \mathcal{C}(X_1,\cdots,X_n)$ closest to $Y$. Write
\[
Z=\sum_{i=1}^n b_iX_i=\sum_{i\in S}b_iX_i \quad \text{ with $b_i>0$ for all $i\in S$}.
\]
By Proposition \ref{prop-det-proj},  $S$ and  the coefficients $(b_i)_{i\in S}$ are uniquely determined by
\begin{align*}
\left\{
  \begin{array}{ll}
    b_i> 0, &  \hbox{for all $i\in S$,}\\
\displaystyle    \langle \sum_{i\in S} b_iX_i, X_j\rangle = \langle Y, X_j\rangle=c_j, & \hbox{for all $j\in S$,}
\\
\displaystyle     \langle \sum_{i\in S} b_iX_i, X_j\rangle \geq \langle Y, X_j\rangle=c_j,&  \text{for all $j \in \{1, 2,\cdots, n\}$}.
  \end{array}
\right.
\end{align*}
In other words, $S$ and $(b_j)_{j\in S}$ are uniquely determined  by
\begin{align*}
\left\{
  \begin{array}{ll}
    b_i> 0, &  \hbox{for all $i\in S$,}\\
   \displaystyle  \sum_{i\in S}a_{ij}b_i=c_j, & \hbox{for all $j\in S$,} \\
   \displaystyle  \sum_{i\in S}a_{ij}b_i\geq c_j, & \hbox{for all $j\in \{1, 2, \cdots, n\}$.}
  \end{array}
\right.
\end{align*}
By noting $a_{ij}=a_{ji}$, we complete the whole proof.
\end{proof}

\section{Applications in function theory}

\subsection{Power functions and signed power functions}

\begin{proof}[Proof of Proposition \ref{prop-h-alpha}]
Let $\alpha\in[0,2), m\in \mathbb{N}$ and  $(a_m,b_m)$ be defined as \eqref{a-b-m}.  By Proposition \ref{prop-det-proj}, for proving the equality \eqref{eq-met-proj}, it suffices to verify
\begin{equation}\label{E:inequlity}
\Big\langle a_mt^{2m}+b_mt^{2m+2}, t^j \Big\rangle_{L^2([-1,1])} \geq \Big\langle |t|^{2m+\alpha}, t^j\Big\rangle_{L^2([-1,1])}\quad \text{for all $j\in\mathbb{N}$}
\end{equation}
and
\begin{equation}\label{E:equality}
\Big\langle a_mt^{2m}+b_mt^{2m+2}, t^j\Big\rangle_{L^2([-1,1])} =\Big\langle |t|^{2m+\alpha}, t^j\Big\rangle_{L^2([-1,1])} \quad \text{for $j\in \{2m,2m+2\}$}.
\end{equation}
If $j$ is an odd number, then \eqref{E:inequlity} holds since both sides of \eqref{E:inequlity} vanish; the same is true for \eqref{E:equality}. So we now focus on even numbers $j  = 2k$ with $k \ge 0$.
Note that $(a_m,b_m)$ defined by \eqref{a-b-m} is in fact the solution of the linear equation
\begin{equation}\label{E:linearequation}
\left(
  \begin{array}{cc}
    \frac{2}{4m+1} & \frac{2}{4m+3} \vspace{2mm}\\
    \frac{2}{4m+3} & \frac{2}{4m+5}
  \end{array}
\right)
\left(
  \begin{array}{c}
    a_m \vspace{2mm}\\
    b_m
  \end{array}
\right)
=\left(
   \begin{array}{c}
     \frac{2}{4m+1+\alpha} \vspace{2mm}\\
     \frac{2}{4m+3+\alpha}
   \end{array}
 \right).
\end{equation}
Since for any even number $j= 2k$, we have
\[
\Big\langle |t|^\beta,t^{2k}\Big\rangle_{L^2([-1,1])}=\frac{2}{\beta+2k+1}\quad \text{for all $\beta\geq 0$}
\]
and
\[
\Big\langle t^{2\ell},t^{2k}\Big\rangle_{L^2([-1,1])}=\frac{2}{2 \ell+2k+1}\quad \text{for all $k, \ell \in \N$},
\]
the equality  \eqref{E:linearequation} is equivalent to the equality \eqref{E:equality}. It remains to show the inequalities \eqref{E:inequlity} for all even numbers $j=2k$ with $k\ge 0$. That is, we need to show
\[
\frac{2 a_m}{2m+2k+1}+\frac{ 2 b_m}{2m+2k+3}\geq \frac{2}{2m+2k+\alpha+1}\quad \text{for all $ k\in\mathbb{N}$}.
\]
Set
\begin{align*}
D_k:=(4m+1+\alpha)(4m+3+\alpha)(2m+2k+1)(2m+2k+3)(2m+2k+\alpha+1)\\
\times \bigg(\frac{2a_m}{2m+2k+1}+\frac{2b_m}{2m+2k+3}-\frac{2}{2m+2k+\alpha+1}\bigg).
\end{align*}
Then we only need to  show that
\begin{align}\label{D-k-pos}
D_k\geq 0 \quad \text{for all $k\in\mathbb{N}$}.
\end{align}
Write
\begin{align}\label{t-x-def}
\tau=4m+3  \an x=4k,
\end{align}
 then
\begin{align*}
D_k= & (2-\alpha)(4m+1)(4m+3)(2m+2k+3)(2m+2k+\alpha+1)\\
& + \alpha(4m+3)(4m+5)(2m+2k+1)(2m+2k+\alpha+1)\\
&-2(4m+1+\alpha)(4m+3+\alpha)(2m+2k+1)(2m+2k+\alpha+1)\\
=& (2-\alpha)(\tau-2)\tau\bigg(\frac{\tau+3}{2}+\frac{x}{2}\bigg)\bigg(\frac{\tau-3}{2}+\frac{x}{2}+\alpha+1\bigg)\\
&+\alpha \tau(\tau+2)\bigg(\frac{\tau-3}{2}+\frac{x}{2}+1\bigg)\bigg(\frac{\tau-3}{2}+\frac{x}{2}+\alpha+1\bigg)\\
&-2(\tau-2+\alpha)(\tau+\alpha)\bigg(\frac{\tau-3}{2}+\frac{x}{2}+1\bigg)\bigg(\frac{\tau-3}{2}+\frac{x}{2}+3\bigg)\\
= & \frac{1}{4}(Ax^2+Bx+C),
\end{align*}
where
\[
A=2\alpha(2-\alpha), B=4\alpha(\alpha-2)(\tau-1) \an C=2\alpha(2-\alpha)(\tau^2-2\tau-3).
\]
Therefore,
\begin{align*}
D_k=\frac{1}{2}\alpha (2-\alpha)[(x-(\tau-1))^2-4].
\end{align*}
By substituting \eqref{t-x-def} into the above equality, we have
\[
D_k=2\alpha(2-\alpha)[(2k-2m-1)^2-1].
\]
By observing
\[
 (2k-2m-1)^2-1\geq 0 \quad \text{for any $k,m\in\mathbb{N}$}
\]
 and using the assumption $\alpha \in [0, 2)$, we obtain the desired inequalities \eqref{D-k-pos}.  This completes the proof of the equality \eqref{eq-met-proj}.

Now we proceed to the proof of the equality \eqref{best-dist}. We have
\begin{align*}
[d(h_{2m+\alpha}, \mathcal{A}_{+})]^2   =&  \int_{[-1,1]} ( |t|^{2m+\alpha} - a_m t^{2m} - b_m t^{2m +2})^2 dt
\\
=& 2 \int_{[0,1]} ( t^{2m+\alpha} - a_m t^{2m} - b_m t^{2m +2})^2 dt
\\
= &  \frac{2}{4m + 2\alpha +1}   +  \frac{2 a_m^2}{4m+1} + \frac{2 b_m^2}{4m+5}
\\
& +  \frac{4 a_m b_m}{4m+3}  - \frac{4 a_m}{4m+\alpha + 1} - \frac{4 b_m}{4m+\alpha +3}.
\end{align*}
By substituting $a_m, b_m$ defined in \eqref{a-b-m} and by writing again $\tau= 4m+3$, we obtain
\begin{align*}
[d(h_{2m+\alpha}, \mathcal{A}_{+})]^2 =&  \frac{2}{\tau +2\alpha-2} + \frac{2}{\tau-2} \Big(\frac{(\tau-2)\tau}{(\tau+\alpha-2)(\tau+\alpha)}\frac{2-\alpha}{2}\Big)^2
\\
&  + \frac{2}{\tau+2} \Big( \frac{\tau(\tau+2)}{(\tau+\alpha-2)(\tau+\alpha)}\frac{\alpha}{2}\Big)^2
\\
& + \frac{4}{\tau} \Big(\frac{(\tau-2)\tau}{(\tau+\alpha-2)(\tau+\alpha)}\frac{2-\alpha}{2}\Big)  \Big( \frac{\tau(\tau+2)}{(\tau+\alpha-2)(\tau+\alpha)}\frac{\alpha}{2}\Big)
\\
&- \frac{4}{\tau+\alpha -2}  \Big(\frac{(\tau-2)\tau}{(\tau+\alpha-2)(\tau+\alpha)}\frac{2-\alpha}{2}\Big)   - \frac{4}{\tau+\alpha} \Big( \frac{\tau(\tau+2)}{(\tau+\alpha-2)(t+\alpha)}\frac{\alpha}{2}\Big)
\\
 = & \frac{2}{\tau +2\alpha-2} +   \frac{2\tau}{(\tau +\alpha -2)^2(\tau +\alpha)^2}  H(\alpha, \tau),
\end{align*}
where
\begin{align*}
H(\alpha, \tau) =& \frac{(2-\alpha)^2}{4} \tau(\tau-2)  + \frac{\alpha^2}{4} \tau(\tau+2) + \frac{\alpha(2 - \alpha)}{2} (\tau-2)(\tau+2)
\\
& - (2-\alpha)(\tau-2)(\tau+\alpha) - \alpha (\tau+2)(\tau + \alpha -2)
\\
 = &-\tau^2 + (2 - 2\alpha )\tau + 2 \alpha (2-\alpha) = -\tau(\tau + 2 \alpha -2) + 2 \alpha (2 - \alpha).
\end{align*}
Therefore, we have
\begin{align*}
[d(h_{2m+\alpha}, \mathcal{A}_{+})]^2 =&   \frac{2}{( \tau +2\alpha-2)(\tau +\alpha -2)^2(\tau+\alpha)^2} K(\alpha, \tau)
\end{align*}
with $K(\alpha, \tau)$ given by
\begin{align*}
K(\alpha, \tau) & = (\tau+\alpha -2)^2(\tau +\alpha)^2 + \tau (\tau+2\alpha-2) H(\alpha, \tau)
\\
& = [ \tau( \tau + 2\alpha -2) - \alpha (2-\alpha)]^2 - \tau^2(\tau + 2 \alpha -2)^2 + 2 \alpha (2 - \alpha) \tau (\tau + 2 \alpha -2)
\\
&= \alpha^2 (2 - \alpha)^2.
\end{align*}
Thus we obtain
\begin{align*}
[d(h_{2m+\alpha}, \mathcal{A}_{+})]^2 =&   \frac{2 \alpha^2 (2 - \alpha)^2}{(\tau +2\alpha-2)(\tau +\alpha -2)^2(\tau +\alpha)^2}
\\
=&  \frac{2 \alpha^2 (2 - \alpha)^2}{(4m +2\alpha+1)(4m +\alpha +1)^2(4m +\alpha+3)^2}.
\end{align*}
This completes the proof of the equality \eqref{best-dist}.
\end{proof}

\begin{proof}[Proof of Corollary \ref{cor-int}]
Fix $m\in\N$ and a positive Radon measure $\nu$  on $[0,2]$. Set
\[
g_{m,\nu}: = \int_{[0,2]} h_{2m+\alpha} d\nu(\alpha)
\]
and
\[
A_{m, \nu} := \int_{[0,2]} a_m(\alpha) d\nu(\alpha) \ge 0 , \, B_{m,\nu}: = \int_{[0,2]} b_m(\alpha) d\nu(\alpha) \ge 0.
\]
By integrating the inequalities \eqref{E:inequlity} and the equalities  \eqref{E:equality} against the measure $\nu$, we obtain
\begin{align*}
\Big\langle A_{m,\nu} t^{2m}+B_{m,\nu}t^{2m+2}, t^j \Big\rangle_{L^2([-1,1])} &\geq \Big\langle g_{m,\nu}, t^j\Big\rangle_{L^2([-1,1])}\quad \text{for all $j\in\mathbb{N}$}
\end{align*}
and
\begin{align*}
\Big\langle A_{m,\nu}t^{2m}+B_{m,\nu}t^{2m+2}, t^j\Big\rangle_{L^2([-1,1])} &=\Big\langle g_{m,\nu}, t^j\Big\rangle_{L^2([-1,1])} \quad \text{for $j\in \{2m,2m+2\}$}.
\end{align*}
Thus by Proposition \ref{prop-det-proj}, we obtain the desired equality
\[
P_{\mathcal{A}_{+}}(g_{m,\nu})= A_{m,\nu} t^{2m} + B_{m,\nu} t^{2m+2}.
\]
\end{proof}

\begin{proof}[Proof of Proposition \ref{prop-f-alpha}]
Let $\alpha\in[0,2), m\in \mathbb{N}$ and $(c_m,d_m)$ be defined as \eqref{c-d-m}.  By Proposition \ref{prop-det-proj}, for proving the equality \eqref{eq-sgn-proj}, it suffices to verify
\begin{equation}\label{E:ineq-sgn}
\Big\langle c_mt^{2m+1}+d_mt^{2m+3}, t^j \Big\rangle_{L^2([-1,1])} \geq \Big\langle  \sgn(t) |t|^{2m+1+\alpha}, t^j\Big\rangle_{L^2([-1,1])}\quad \text{for all $j\in\mathbb{N}$}
\end{equation}
and
\begin{equation}\label{E:eq-sgn}
\Big\langle c_mt^{2m+1}+d_mt^{2m+3}, t^j\Big\rangle_{L^2([-1,1])} =\Big\langle \sgn(t) |t|^{2m+1+\alpha}, t^j\Big\rangle_{L^2([-1,1])} \, \text{for $j\in \{2m+1,2m+3\}$}.
\end{equation}
If $j$ is an even number, then  \eqref{E:ineq-sgn} holds since both sides of \eqref{E:ineq-sgn} vanish; the same is true for \eqref{E:eq-sgn}. So we now focus on odd numbers $j  = 2k+1$ with $k\ge 0$.
Note that $(c_m,d_m)$ defined by \eqref{c-d-m} is in fact the solution of the linear equation
\begin{align}\label{linear-eq-sgn}
\left(
  \begin{array}{cc}
    \frac{2}{4m+3} & \frac{2}{4m+5}  \vspace{2mm}\\
    \frac{2}{4m+5} & \frac{2}{4m+7}\\
  \end{array}
\right)
\left(
  \begin{array}{c}
    c_m \vspace{2mm}
\\
    d_m
  \end{array}
\right)
=\left(
   \begin{array}{c}
     \frac{2}{4m+\alpha+3} \vspace{2mm}
 \\
     \frac{2}{4m+\alpha+5}
   \end{array}
 \right).
\end{align}
Since for any odd number $j= 2k+1$,
\[
\Big\langle \sgn(t) |t|^\gamma,t^{2k+1}\Big\rangle_{L^2([-1,1])}=\frac{2}{\gamma+2k+2}\quad \text{for all $\gamma \geq 0$}
\]
and
\[
\Big\langle t^{2\ell+1},t^{2k+1}\Big\rangle_{L^2([-1,1])}=\frac{2}{2 \ell+2k+3}\quad \text{for all $k, \ell \in \N$},
\]
the equality  \eqref{linear-eq-sgn} is equivalent to the equality \eqref{E:eq-sgn}. It remains to show the inequalities \eqref{E:ineq-sgn} for all odd numbers $j=2k+1$ with $k\ge 0$. That is, we need to show
\[
\frac{2 c_m}{2m+2k+3}+\frac{2 d_m}{2m+2k+5}\geq \frac{2}{2m+2k+\alpha+3}\quad \text{for all $ k\in\mathbb{N}$}.
\]
Set
\begin{align*}
T_k:=(2m+2k+3)(2m+2k+5)(4m+3+\alpha)(4m+5+\alpha)(2m+2k+3+\alpha)\\
\times \bigg(
\frac{2c_m}{2m+2k+3}+\frac{2d_m}{2m+2k+5}-\frac{2}{2m+2k+3+\alpha}
\bigg).
\end{align*}
Then we only need to  show
\begin{align}\label{T-k-pos}
T_k\geq 0 \quad \text{for all $k\in\mathbb{N}$}.
\end{align}
By using exactly the same arguments as those in dealing with $D_k$ in the proof of Proposition \ref{prop-h-alpha} (another simpler way  is to replace everywhere the pair $(m, k)$ in the definition of $D_k$ by the pair $(m+\frac{1}{2}, k+\frac{1}{2})$ to obtain a reduced form of $T_k$), we obtain
\begin{align*}
T_k=2\alpha(2-\alpha)[(2k-2m-1)^2-1] \quad \text{for all $k\in \N$}.
\end{align*}
By observing $(2k-2m-1)^2-1\geq 0$ for any $k,m\in\mathbb{N}$ and using the assumption $\alpha \in [0, 2)$, we obtain the desired inequalities \eqref{T-k-pos}.  This completes the proof of the equality \eqref{eq-sgn-proj}.

The verification of the equality \eqref{best-dist-sgn} is similar to that of the equality \eqref{best-dist}.
\end{proof}

\subsection{Indicator functions}
For any $n\in \N$, define an $(n+1)\times (n+1)$-matrix by
\[
M_n  := \Big(\langle t^i, t^j\rangle_{L^2([-1,1])}\Big)_{0\le i, j \le n}.
\]
By the linear independence of the functions $1, t, t^2, \cdots$ on $[-1,1]$, for  any $n\in \N$, the matrix $M_n$ is non-singular.

Note that for any $n\in \N$ and  any $a\in [-1,1)$, we have
\[
\langle \psi_a, t^n\rangle_{L^2([-1,1])} = \int_a^1 t^n dt = \frac{1 - a^{n+1}}{n+1}.
\]
Let $v_{a}^{(n)}\in \R^{n+1}$ be the column vector defined by
\begin{align*}
v_{a}^{(n)} & = (\langle \psi_a, 1\rangle_{L^2([-1,1])}, \, \langle \psi_a, t \rangle_{L^2([-1,1])}, \, \cdots,\, \langle \psi_a, t^n\rangle_{L^2([-1,1])} )^\top
\\
& = \Big(1 - a, \frac{1-a^2}{2}, \cdots, \frac{1- a^{n+1}}{n+1} \Big)^\top.
\end{align*}

Denote by $\R_{+}^* = (0, \infty)$ the set of all positive numbers.  Lemmas \ref{lem-33} and \ref{lem-inf-ineq} below will be used in the proof of Propositions \ref{prop-indicator} and \ref{prop-neg-a}.

\begin{lemma}\label{lem-33}
Assume that $a\in [-1, 1)$. Then the linear equation
\begin{align}\label{eq-33}
M_2 x = v_{a}^{(2)}
\end{align}
has a solution in $(\R_{+}^*)^3$ if and only if
\[
0< a < \frac{\sqrt{105}-5}{10}.
\]
\end{lemma}
\begin{proof}
The solution $x= (x_0, x_1, x_2)$ of the linear equation \eqref{eq-33} is given by
\begin{align}\label{sol-x}
\left\{
\begin{array}{l}
x_0= \frac{1}{8} (4-9a+5a^3) = \frac{1}{8} (1 -a)(4-5a-5a^2)
\vspace{2mm}
\\
x_1= \frac{3}{4}(1 - a^2)
\vspace{2mm}
\\
x_2  =\frac{15}{8} (a-a^3) = \frac{15}{8}a (1-a^2)
\end{array}
\right..
\end{align}
A simple computation shows that,  under the assumption $a\in [-1, 1)$, the solution $x$ belongs to $(\R_+^*)^3$ if and only if $0< a < \frac{\sqrt{105}-5}{10}$.
This completes the proof of the lemma.
\end{proof}

\begin{lemma}\label{lem-inf-ineq}
Assume that $\rho \in [0,1)$. Then the condition
\begin{align}\label{rho-inf-n}
 (1 -3\rho+2\rho^{n+1})n  + 3 \rho^{n+1} \ge  3 \rho \quad  \text{for all integers $n \ge 1$}
\end{align}
holds if and only if   $\rho\in [0, \frac{1}{5}]$.
\end{lemma}
\begin{proof}
Assume that \eqref{rho-inf-n} holds. Then by taking $n=1$, we have
\[
1 - 3 \rho + 2 \rho^2 + 3 \rho^2 \ge 3 \rho.
\]
This combined with the assumption $\rho \in [0,1)$ implies that $\rho \in [0,  1/5]$.

Conversely, assume that $\rho \in [0, 1/5]$. Then for $n=1$, we have
\[
1 -3\rho+2\rho^{2} + 3 \rho^{2} -  3 \rho  = (1-\rho)(1 - 5\rho)\ge 0.
\]
This implies the inequality \eqref{rho-inf-n} for $n=1$.
Now  assume that $n\ge 2$, we have
\begin{align*}
(1 -3\rho+2\rho^{n+1})n  + 3 \rho^{n+1} -  3 \rho \ge (1 - \frac{3}{5} ) n -  \frac{3}{5} = \frac{2n-3}{5} \ge \frac{1}{5} \ge 0.
\end{align*}
This implies the inequality \eqref{rho-inf-n} for all integers $n\ge 2$.
\end{proof}

\begin{proof}[Proof of Proposition \ref{prop-indicator}]
Let  $x_0, x_1, x_2$ be given as in \eqref{sol-x} and recall that $x$ is the solution to the linear equation \eqref{eq-33}. By the discussions in \S\ref{sec-scheme} and Proposition~\ref{prop-det-proj},  the equality  \eqref{gamma-set-3} holds if and only if the following conditions are all satisfied:
\begin{itemize}
\item $x\in (\R_{+}^*)^3$;
\item for  all $j\in \{0,1,2\}$,
\begin{align}\label{j-0-2}
\Big\langle \sum_{i=0}^2 x_i t^i, t^j\Big\rangle_{L^2([-1,1])} =\langle \psi_a, t^j\rangle_{L^2([-1,1])} =  \frac{1 - a^{j+1}}{j+1};
\end{align}
\item for all integers $j\ge 3$,
\begin{align}\label{j-ge-3}
\Big\langle \sum_{i=0}^2 x_i t^i, t^j\Big\rangle_{L^2([-1,1])} \ge \langle \psi_a, t^j\rangle_{L^2([-1,1])} =   \frac{1 - a^{j+1}}{j+1}.
\end{align}
\end{itemize}
Note that the system of linear equations \eqref{j-0-2} for $j\in \{0, 1, 2\}$ is equivalent to the linear equation \eqref{eq-33}. Thus by the definition of $x$, the equalities \eqref{j-0-2} hold  for  all $j\in \{0, 1, 2\}$. By Lemma \ref{lem-33}, $x\in (\R_{+}^*)^3$ if and only if $0< a < \frac{\sqrt{105}-5}{10}$.  Now let us analyze \eqref{j-ge-3}. If $j= 2n+1$ with $n \ge 1$, then
\begin{align*}
& \Big\langle \sum_{i=0}^2 x_i t^i, t^{2n+1}\Big\rangle_{L^2([-1,1])} - \langle \psi_a, t^{2n+1}\rangle_{L^2([-1,1])}
= \frac{2}{2n+3} \cdot \frac{3}{4} (1 - a^2) - \frac{1 -a^{2n+2}}{2n+2}
\\
& = \frac{(1 -3a^2+2a^{2n+2})n  + 3 a^{2n+2} - 3a^2}{2(2n+3)(n+1)}
\end{align*}
and if $j = 2n+2$ with $n \ge 1$, then
\begin{align*}
& \Big\langle \sum_{i=0}^2 x_i t^i, t^{2n+2}\Big\rangle_{L^2([-1,1])} - \langle \psi_a, t^{2n+2}\rangle_{L^2([-1,1])}
\\
=&  \frac{2}{2n+3}  \cdot \frac{1}{8}(4-9a + 5a^3)    + \frac{2}{2n+5} \cdot \frac{15}{8}(a-a^3)- \frac{1 -a^{2n+3}}{2n+3}
\\
= & \frac{a \Big[  (3-5a^2+2a^{2n+2}) n - 5 a^2 +5 a^{2n+2}\Big]}{(2n+3)(2n+5)}.
\end{align*}
Therefore, the equality  \eqref{gamma-set-3} holds if and only if the following conditions are all satisfied:
\begin{itemize}
\item $0 < a < \frac{\sqrt{105}-5}{10}$;
\item for all integers $n\ge 1$,
\[
(1 -3a^2+2a^{2n+2})n  + 3 a^{2n+2} - 3a^2  \ge 0;
\]
\item for all integers $n\ge 1$,
\[
 a \Big[  (3-5a^2+2a^{2n+2}) n - 5 a^2 +5 a^{2n+2}\Big] \ge 0.
\]
\end{itemize}
Observe that
\begin{align*}
(3-5a^2+2a^{2n+2})n  +5 a^{2n+2} - 5 a^2 =&  3\Big[ (1 -3a^2+2a^{2n+2})n  + 3 a^{2n+2} -  3a^2\Big]
\\
& + 4 a^2 ( n +1)(1 -  a^{2n})
\\
 \ge &  3\Big[ (1 -3a^2+2a^{2n+2})n  + 3 a^{2n+2} -  3a^2\Big].
\end{align*}
Thus by Lemma \ref{lem-inf-ineq} and  the inequality
$\frac{1}{\sqrt{5}} < \frac{\sqrt{105}-5}{10}$,
 the equality  \eqref{gamma-set-3} holds if and only if  $0< a \le \frac{1}{\sqrt{5}}$.

Finally, by Proposition \ref{prop-det-proj}, the above arguments imply that for $0< a \le \frac{1}{\sqrt{5}}$, we have
\[
P_{\mathcal{A}_{+}} (\psi_a) = \sum_{i=0}^2 x_i t^i.
\]
This is the desired equality \eqref{P-C-psi-a}.
\end{proof}

\begin{proof}[Proof of Corollary \ref{cor-monotone}]
The proof of Corollary \ref{cor-monotone} is similar to that of Corollary \ref{cor-int}.
\end{proof}

\begin{proof}[Proof of Proposition \ref{prop-neg-a}]
Set
\[
y_0 := \frac{1-a}{2}, \quad y_1 : = \frac{3}{4} (1 -a^2).
\]
Then the assumption $a\in [-1, 1)$ implies that $y= (y_0, y_1)\in (\R_{+}^*)^2$.  Therefore, by  Proposition~\ref{prop-det-proj},   the equality \eqref{P-neg-a} holds if and only if
\begin{align}\label{j-0-1}
\Big\langle \sum_{i=0}^1 y_i t^i, t^j\Big\rangle_{L^2([-1,1])} =\langle \psi_a, t^j\rangle_{L^2([-1,1])} =   \frac{1 - a^{j+1}}{j+1} \quad \text{for $j\in \{0,1\}$}
\end{align}
and
\begin{align}\label{j-ge-2}
\Big\langle \sum_{i=0}^1 y_i t^i, t^j\Big\rangle_{L^2([-1,1])} \ge \langle \psi_a, t^j\rangle_{L^2([-1,1])} =   \frac{1 - a^{j+1}}{j+1} \quad \text{for $j\ge 2$}.
\end{align}
The equalities \eqref{j-0-1} can be checked directly by using the definitions of $y_0$ and $y_1$.  Now let us analyze the inequalities \eqref{j-ge-2}. Note that if $j = 2n$ for $n\ge 1$, then
\begin{align*}
\Big\langle \sum_{i=0}^1 y_i t^{i}, t^{2n}\Big\rangle_{L^2([-1,1])}  -  \langle \psi_a, t^{2n}\rangle_{L^2([-1,1])} &=   \frac{2}{2n+1} \frac{1-a}{2} - \frac{1-a^{2n+1}}{2n+1}= \frac{a^{2n+1}-a}{2n+1}
\end{align*}
and if $j = 2n+1$ with $n\ge 1$, then
\begin{align*}
\Big\langle \sum_{i=0}^1 y_i t^{i}, t^{2n+1}\Big\rangle_{L^2([-1,1])}  -  \langle \psi_a, t^{2n+1}\rangle_{L^2([-1,1])} &=   \frac{2}{2n+3}  \cdot \frac{3}{4}(1-a^2) - \frac{1-a^{2n+2}}{2n+2}
\\
& = \frac{(1 -3a^2+2a^{2n+2})n  + 3 a^{2n+2} - 3a^2}{2(2n+3)(n+1)}.
\end{align*}
Therefore,    \eqref{j-ge-2} holds if and only if
\begin{align}\label{simple-comp}
\left\{
\begin{array}{l}
a^{2n+1} -a \ge 0
\vspace{2mm}
\\
(1 -3a^2+2a^{2n+2})n  + 3 a^{2n+2} \ge  3a^2
\end{array}
\right. \quad \text{for all integers $n\ge 1$}.
\end{align}
By Lemma \ref{lem-inf-ineq},  under the assumption $a\in [-1,1)$, the condition \eqref{simple-comp} holds  if and only if $-\frac{1}{\sqrt{5}}\le a \le 0$. 
This completes the proof of the proposition.
\end{proof}

Lemmas \ref{lem-44} and \ref{lem-two-ineq} below will be used in the proof of Proposition \ref{prop-big-a}.

\begin{lemma}\label{lem-44}
Assume that $a\in [-1, 1)$. Then the solution to the linear equation
\begin{align}\label{eq-44}
M_3 z = v_{a}^{(3)}
\end{align}
belongs to $(\R_{+}^*)^4$ if and only if
\[
\frac{1}{\sqrt{5}}< a < \frac{\sqrt{105}-5}{10}.
\]
\end{lemma}

\begin{proof}
The solution $z= (z_0, z_1, z_2, z_3)$ of the linear equation \eqref{eq-44} is given by
\begin{align}\label{sol-z}
\left\{
\begin{array}{l}
z_0 = \frac{1}{8} (1 -a)(4-5a-5a^2)
\vspace{2mm}
\\
z_1= \frac{15}{32} (1 - a^2) (3-7a^2)
\vspace{2mm}
\\
z_2  = \frac{15}{8}a (1-a^2)
\vspace{2mm}
\\
z_3 = \frac{35}{32}(1 -a^2)(5a^2 -1)
\end{array}
\right..
\end{align}
Therefore,   under the assumption $a\in [-1, 1)$, the solution $z$ belongs to $(\R_+^*)^4$ if and only if
$
\frac{1}{\sqrt{5}}< a < \frac{\sqrt{105}-5}{10}.
$
This completes the proof of the lemma.
\end{proof}

\begin{lemma}\label{lem-two-ineq}
Suppose that $\frac{1}{\sqrt{5}}< a < \frac{\sqrt{105}-5}{10}$. Then for all integers $n\ge 2$, we have
\[
(3a -5a^3+2a^{2n+1})n + 3 a^{2n+1}-3a\ge 0
\]
and
\[
(-3 + 30a^2 -35 a^4 + 8 a^{2n+2})n^2 + (3 -35 a^4 +32 a^{2n+2})n +30 a^{2n+2} -30 a^2 \ge 0.
\]
\end{lemma}
\begin{proof}
Assume that $\frac{1}{\sqrt{5}}< a < \frac{\sqrt{105}-5}{10}$. For the first inequality, note that, combined with the elementary inequality
\[
\Big( \frac{\sqrt{105}-5}{10}\Big)^2 < \frac{3}{10},
\]
the assumption $\frac{1}{\sqrt{5}}< a < \frac{\sqrt{105}-5}{10}$ implies
\[
3a - 5a^3 = a(3-5a^2) > 0 \an 3a -10 a^3 = a(3-10a^2)> 0.
\]
Thus for any $n\ge 2$, we have
\begin{align*}
(3a -5a^3+2a^{2n+1})n + 3 a^{2n+1}-3a &  \ge   (3a - 5a^3) n- 3a
\\
&  \ge 2(3a -5a^3) -3a  = 3a -10a^3>0.
\end{align*}

For the second inequality, note that, combined with the elementary inequality
\[
\Big( \frac{\sqrt{105}-5}{10}\Big)^4 < \frac{3}{35},
\]
the assumption $\frac{1}{\sqrt{5}}< a < \frac{\sqrt{105}-5}{10}$ implies
\[
3 - 35 a^4 > 0
\]
and
\[
-3 + 30 a^2 - 35 a^4 > -3  + 30 \Big(\frac{1}{\sqrt{5}}\Big)^2 -35 a^4 = 3 -35 a^4 > 0.
\]
Note also that
\[
-1 + 15 x^2 -35 x^4 \ge 0 \quad \text{provided that $\frac{15 - \sqrt{85}}{70}\le x^2\le \frac{15 + \sqrt{85}}{70}$}.
\]
Therefore, by noting
\[
 \frac{15 - \sqrt{85}}{70} \le \frac{1}{5} \le  a^2 \le \Big( \frac{\sqrt{105}-5}{10}\Big)^2 \le \frac{15 + \sqrt{85}}{70},
\]
we have
\[
-1 +15 a^2 -35 a^4\ge 0.
\]
It follows that for any integer $n\ge 2$, we have
\begin{align*}
& (-3 + 30a^2 -35 a^4 + 8 a^{2n+2})n^2 + (3 -35 a^4 +32 a^{2n+2})n +30 a^{2n+2} -30 a^2
\\
& \ge (-3 + 30a^2 -35 a^4)n^2 + (3 -35 a^4)n -30 a^2
\\
& \ge (-3 + 30a^2 -35 a^4)\times 4+ (3 -35 a^4)\times 2 -30 a^2
\\
& = 6 (-1 +15 a^2 -35 a^4) \ge 0.
\end{align*}

The lemma is  proved completely.
\end{proof}

\begin{proof}[Proof of Proposition \ref{prop-big-a}]
Let  $z = (z_0, z_1, z_2, z_3)$ be given as in \eqref{sol-z} and recall that $z$ is the solution to the linear equation \eqref{eq-44}. By the discussions in \S\ref{sec-scheme} and Proposition~\ref{prop-det-proj},  the equality  \eqref{gamma-set-4} holds if and only if the following conditions are all satisfied:
\begin{itemize}
\item $z \in (\R_{+}^*)^4$;
\item for $j\in \{0, 1, 2, 3\}$, we have
\begin{align}\label{j-0-3}
\Big\langle \sum_{i=0}^3 z_i t^i, t^j\Big\rangle_{L^2([-1,1])} =\langle \psi_a, t^j\rangle_{L^2([-1,1])} =   \frac{1 - a^{j+1}}{j+1};
\end{align}
\item for all integers $j \ge 4$, we have
\begin{align}\label{j-ge-4}
\Big\langle \sum_{i=0}^3 z_i t^i, t^j\Big\rangle_{L^2([-1,1])} \ge \langle \psi_a, t^j\rangle_{L^2([-1,1])} =   \frac{1 - a^{j+1}}{j+1}.
\end{align}
\end{itemize}
Note that the system of linear equations \eqref{j-0-3} for $j\in \{0, 1, 2, 3\}$ is equivalent to the linear equation \eqref{eq-44}. Thus by the definition of $z$, the equalities \eqref{j-0-3} hold  for  all $j\in \{0, 1, 2, 3\}$.
By Lemma \ref{lem-44}, $z\in (\R_{+}^*)^4$ if and only if
$
\frac{1}{\sqrt{5}}< a < \frac{\sqrt{105}-5}{10}.
$  Now let us analyze the inequalities \eqref{j-ge-4} for $j \ge 4$. For even numbers $j  = 2n$ with $n\ge 2$, we have
\begin{align*}
& \Big\langle \sum_{i=0}^3 z_i t^{i}, t^{2n}\Big\rangle_{L^2([-1,1])}  -  \langle \psi_a, t^{2n}\rangle_{L^2([-1,1])}
\\
&  = \frac{2}{2n+1} \cdot \frac{1}{8} (1 -a)(4-5a-5a^2) + \frac{2}{2n+3} \cdot  \frac{15}{8}a (1-a^2) - \frac{1  - a^{2n+1}}{2n+1}
\\
&  = \frac{(3a -5a^3+2a^{2n+1})n + 3 a^{2n+1}-3a}{(2n+1)(2n+3)}
\end{align*}
and for odd numbers $j = 2n +1$ with $n\ge 2$, we have
\begin{align*}
&\Big\langle \sum_{i=0}^3 z_i t^{i}, t^{2n+1}\Big\rangle_{L^2([-1,1])}  -  \langle \psi_a, t^{2n+1}\rangle_{L^2([-1,1])}
\\
& = \frac{2}{2n+3} \cdot \frac{15}{32} (1 - a^2) (3-7a^2) + \frac{2}{2n+5} \cdot \frac{35}{32}(1 -a^2)(5a^2 -1) - \frac{1  - a^{2n+2}}{2n+2}
\\
&=  \frac{(-3 + 30a^2 -35 a^4 + 8 a^{2n+2})n^2 + (3 -35 a^4 +32 a^{2n+2})n +30 a^{2n+2} -30 a^2 }{2(2n+3)(2n+5)(2n+2)}.
\end{align*}
Therefore, the equality  \eqref{gamma-set-4} holds if and only if the following conditions are all satisfied:
\begin{itemize}
\item $\frac{1}{\sqrt{5}}< a < \frac{\sqrt{105}-5}{10}$;
\item for all integers $n\ge 2$,
\[
(3a -5a^3+2a^{2n+1})n + 3 a^{2n+1}-3a \ge 0;
\]
\item for all integers $n\ge 2$,
\[
(-3 + 30a^2 -35 a^4 + 8 a^{2n+2})n^2 + (3 -35 a^4 +32 a^{2n+2})n +30 a^{2n+2} -30 a^2\ge 0.
\]
\end{itemize}
Thus by Lemma \ref{lem-two-ineq},  the equality  \eqref{gamma-set-4} holds if and only if  $\frac{1}{\sqrt{5}}< a < \frac{\sqrt{105}-5}{10}$.

Finally, by Proposition \ref{prop-det-proj}, the above arguments imply that if $\frac{1}{\sqrt{5}}< a < \frac{\sqrt{105}-5}{10}$, then
\[
P_{\mathcal{A}_{+}} (\psi_a) = \sum_{i=0}^3 z_i t^i.
\]
This is the desired equality \eqref{P-big-a}.
\end{proof}

\section{Appendix}

In this appendix, we show that the set
\[
\mathcal{C}_{L^2([-1,0])}[[ \{t^n\}_{n=0}^\infty]] = \bigg\{\sum_{n=0}^\infty a_nt^n\Big|a_n\geq 0,\text{ the series }\sum_{n=0}^\infty a_nt^n\text{ converges in }L^2([-1,0])\bigg\}
\]
is not closed. Or equivalently, we show that the set
\[
\mathcal{C}_{L^2([0, 1])}[[ \{(-t)^n\}_{n=0}^\infty]] = \bigg\{\sum_{n=0}^\infty a_n (-t)^n\Big|a_n\geq 0,\text{ the series }\sum_{n=0}^\infty a_n(-t)^n\text{ converges in }L^2([0,1])\bigg\}
\]
is not closed.  Indeed, set $\rho_k = 1 - \frac{1}{k+1}$ for any $k\in \N$. Let
\[
g_k(t): =  \frac{1}{(1 + \rho_k t)^2} = \sum_{n=0}^\infty   (n+1)  \rho_k^n (-t)^n, \quad t \in [0, 1).
\]
Then clearly, we have $g_k\in \mathcal{C}_{L^2([0, 1])}[[ \{(-t)^n\}_{n=0}^\infty]] $ and
\[
g_k(t) \xrightarrow[k\to\infty]{\text{in $L^2([0,1])$}}g_\infty(t)= \frac{1}{(1 + t)^2}.
\]
Now let us show that $g_\infty \notin \mathcal{C}_{L^2([0, 1])}[[ \{(-t)^n\}_{n=0}^\infty]]$. Otherwise, $g_\infty \in \mathcal{C}_{L^2([0, 1])}[[ \{(-t)^n\}_{n=0}^\infty]]$, then there exists  a sequence $(a_n)_{n = 0}^\infty$ of non-negative numbers such that
\[
g_\infty = \sum_{n=0}^\infty a_n (-t)^n,
\]
where the equality is understood as
\begin{align}\label{g-cv-series}
\lim_{N\to\infty} \Big \| \sum_{n=0}^N a_n(-t)^n - \frac{1}{(1 + t)^2} \Big\|_{L^2([0,1])} = 0.
\end{align}
The above convergence implies
\begin{align}\label{bdd-an}
 \frac{a_n^2 }{2n + 1}    = \| a_n (-t)^n \|_{L^2([0,1])}^2 \le \sup_n \| a_n (-t)^n \|_{L^2([0,1])}^2  = M<\infty.
\end{align}

Since the $L^2$-norm convergence implies the almost everywhere convergence along a subsequence, \eqref{g-cv-series} implies that, along a subsequence $(N_k)_{k=0}^\infty$ of positive numbers, we have
\[
\lim_{k\to\infty}\sum_{n=0}^{N_k} a_n (-t)^n = \frac{1}{(1 + t)^2} \quad \text{for Lebesgue a.e. $t\in [0,1]$}.
\]
Note that  \eqref{bdd-an} implies that the series $\sum_{n} a_n z^n$ has a radius of convergence not smaller than $1$, hence we have
\[
\lim_{k\to\infty}\sum_{n=0}^{N_k} a_n (-t)^n  =  \sum_{n=0}^{\infty} a_n (-t)^n \quad \text{for all $t \in [0,1)$}.
\]
Therefore, we obtain
\[
 \sum_{n=0}^{\infty} a_n (-t)^n = \frac{1}{(1 + t)^2} \quad \text{for Lebesgue a.e. $t\in [0,1]$}.
\]
By elementary results on analytic funtions, the above equality implies that $a_n =  n+1$ for all $n\in \N$.  Thus the limit relation \eqref{g-cv-series} now reads as
\[
\lim_{N\to\infty} \Big\|  \sum_{n = 0}^{N} (n+1) (-t)^n   - \frac{1}{(1 + t)^2} \Big\|_{L^2([0,1])} = 0.
\]
However, for any large integer $N$ and $t\in (0,1)$, we have
\begin{align*}
 \sum_{n = 0}^{N} (n+1) (-t)^n   - \frac{1}{(1 + t)^2}  = \sum_{n = N+1}^{\infty} (n+1) (-t)^n
 =  \frac{(N+2) (-t)^{N+1}}{1 + t} +  \frac{(-t)^{N+2}}{(1 + t)^2}.
\end{align*}
Since
\[
\Big\|   \frac{(N+2) (-t)^{N+1}}{1 + t}\Big\|_{L^2([0,1])}  \ge \frac{N+2}{2}  \| (-t)^{N+1}\|_{L^2([0,1])} = \frac{N+2}{2 \sqrt{2N+3}}
\]
and
\[
\Big\|  \frac{(-t)^{N+2}}{(1 + t)^2} \Big\|_{L^2([0,1])}  \le 1,
\]
we have
\[
\lim_{N\to\infty}\Big\|  \sum_{n = 0}^{N} (n+1) (-t)^n   - \frac{1}{(1 + t)^2} \Big\|_{L^2([0,1])} \ge  \lim_{N\to\infty}  \Big( \frac{N+2}{2 \sqrt{2N+3}} -1\Big) = \infty.
\]
Thus we obtain a contradiction. Hence  $\mathcal{C}_{L^2([0, 1])}[[ \{(-t)^n\}_{n=0}^\infty]]$ is not closed in $L^2([0,1])$.

%\bibliography{mybib}
%\bibliographystyle{plain}

\def\cprime{$'$} \def\cydot{\leavevmode\raise.4ex\hbox{.}}

\end{document}